\documentclass[11pt]{article}

\RequirePackage{newtxtext}
\RequirePackage{newtxmath}

\usepackage{amsthm}
\usepackage{bm}
\usepackage{endnotes}
\DeclareMathOperator*{\argmax}{arg\,max}
\DeclareMathOperator*{\argmin}{arg\,min}

\usepackage{caption}
\usepackage[labelfont=sf]{subcaption}
\usepackage{algorithm}
\usepackage{algpseudocode}
\usepackage{tikz}
\newtheorem{proposition}{Proposition}
\newtheorem{theorem}{Theorem}

\allowdisplaybreaks

\usepackage{natbib}

\usepackage[margin=1in]{geometry}

\usepackage[colorlinks=true,linkcolor=blue,citecolor=blue,urlcolor=blue]{hyperref}

\begin{document}

\title{Dynamic Transmission Line Switching Amidst Wildfire-Prone Weather Under Decision-Dependent Uncertainty}
\author{
Juan-Alberto Estrada-Garcia\thanks{Department of Industrial and Operations Engineering, University of Michigan, Ann Arbor, MI. \texttt{juanest@umich.edu}}
\and Ruiwei Jiang\thanks{Department of Industrial and Operations Engineering, University of Michigan, Ann Arbor, MI. \texttt{ruiwei@umich.edu}}
\and Alexandre Moreira\thanks{Lawrence Berkeley National Laboratory, Berkeley, CA. \texttt{amoreira@lbl.gov}}
}
\date{}

\maketitle

\begin{abstract}
During dry and windy seasons, environmental conditions significantly increase the risk of wildfires, exposing power grids to disruptions caused by transmission line failures. Wildfire propagation exacerbates grid vulnerability, potentially leading to prolonged power outages. To address this challenge, we propose a multi-stage optimization model that dynamically adjusts transmission grid topology in response to wildfire propagation, aiming to develop an optimal response policy. By accounting for decision-dependent uncertainty, where line survival probabilities depend on usage, we employ distributionally robust optimization to model uncertainty in line survival distributions. We adapt the stochastic nested decomposition algorithm and derive a deterministic upper bound for its finite convergence. To enhance computational efficiency, we exploit the Lagrangian dual problem structure for a faster generation of Lagrangian cuts. 
Using realistic data from the California transmission grid, we demonstrate the superior performance of dynamic response policies against two-stage alternatives through a comprehensive case study. In addition, we construct easy-to-implement policies that significantly reduce computational burden while maintaining good performance in real-time deployment.
\end{abstract}

\maketitle

\section{Introduction}\label{sec:Intro}

Transmission grids are crucial infrastructures for societies worldwide. These grids occupy large physical spaces and are distributed in both urban and rural environments, interacting with risk factors to their continuous operation such as tall vegetation, buildings, and vehicles. Risk factors in environments can cause physical damages to key system components, therefore posing a major challenge for transmission grid operators~\citep{rathor2020energy}. Among these risk factors, wildfires have raised major concerns due to their reported impacts on multiple regions of the world and have increased frequency in recent decades~\citep{halofsky2020changing}. Wind, temperature, and humidity have significant influence in wildfires ignition and propagation, but their occurrence is difficult to predict~\citep{kondylatos2022wildfire}. Wildfires pose significant challenges to transmission grids, because they not only break down infrastructure (e.g., towers and poles;~\cite{dian2019integrating}) but also decrease the conductor ampacity (capacity) of transmission lines~\citep{choobineh2016power}, creating outages and cascading impacts. An example is the 2017 Thomas wildfire in California, where increased temperatures, ash, and fire decreased the operative capabilities of power lines. As a result, the lines were forced to shut down, while other lines increased power flow beyond their nominal capacity, causing power outages \citep{al2022impact}. Likewise, forest wildfires have caused a series of accidents related to power system in China. In February 2010, the Guizhou province power grid experienced 44 line-tripping accidents, most of which were attributed to phase-to-ground faults in wildfire-prone conditions~\citep{pu2015study}.

These incidents result in substantial repair costs and widespread power outages, leaving communities vulnerable during emergencies. The increasing frequency and intensity of wildfires due to climate change require robust strategies to protect power systems and ensure uninterrupted service.

One approach to mitigating the wildfire impacts is transmission line switching, which reconfigures the topology of transmission grids and reroutes the electricity flow around wildfire-prone regions~\citep{moreira2024distribution}. Line switching is considered an important measure in response to events in power systems. For example,~\cite{nagarajan2016optimal} suggests installing switches on transmission lines to enhance grid resiliency,~\cite{bayani2023resilient} applies line switching to hedge against wildfire events, and the California utility San Diego Gas \& Electric installed line switches to better mitigate wildfire impacts~\citep{udren2022managing}. When combined with other planning and restoration strategies, line switching has been shown to significantly enhance the resilience of power grids to the threats of extreme events including wildfires~\citep{abdelmalak2022enhancing,wang2022coordinated}.

Despite its wide applications, the modeling of line switching in wildfire-prone weather poses challenges. First, as wildfires decrease transmission line capacities, the probability of a line failure depends not only on environmental factors (temperature, humidity, wildfires in the vicinity, etc.) but also on the magnitude of power flow on the line, rendering the realization of line failures associated with \emph{decision-dependent uncertainty} (DDU;~\cite{muhs2020characterizing,moreira2024distribution}). Stochastic programs with DDU are computationally intractable~\citep{gupta2011solution,li2021review,ryu2019nurse}, mainly because (i) an accurate model of DDU demands estimating the distribution of uncertainty with respect to \emph{all} possible decisions, and (ii) DDU usually results in a nonconvex optimization model that is difficult to solve, especially for large-scale engineering systems such as transmission grids. Second, the dynamic and stochastic nature of wildfires adds complexity to decision making, as operators must consider the evolving threats. Traditional two-stage models (see, e.g.,~\cite{huang2017electrical,pianco2024decision}) to this problem can provide feasible switching plans, but fall short in capturing the dynamic states of wildfire propagation. In other words, two-stage line switching plans cannot be adjusted according to the most up-to-date wildfire states and hence are suboptimal.

To address these modeling challenges, in this paper we propose a multi-stage distributionally robust optimization (DRO) model. Specifically, we model the (line failure) DDU by adopting an ambiguity set of probability distributions, wherein the probability of a line failure depends parametrically on the magnitude of power flow passing through the line. This modeling choice waives the (burdensome) need to estimate the line failure probability for each power flow magnitude and, in addition, alleviates the potential misspecification of the line failure probabilities. This framework ensures that the ensuing line switching policy remains effective across a range of plausible scenarios, enhancing the robustness and reliability of power system operations under wildfire threats. Furthermore, we characterize the wildfire propagation using a scenario tree, which models the dynamic range and severity of the wildfire and accounts for its temporal and spatial variability. To mitigate the computational challenges arising from the large number of decision and state variables, we adapt the stochastic nested decomposition (SND) algorithm~\citep{zou2018multistage,zou2019stochastic,yu2022multistage}, derive deterministic upper bound for its finite convergence, and produce tight Lagrangian cuts by reusing past Lagrangian multipliers.

Our contributions in this work are fourfold. Firstly, we propose one of the first dynamic topology optimization models amidst wildfire risks as a multi-stage DRO formulation with DDU. Secondly, to solve this problem, we extend the SND algorithm to solve our DRO model with DDU. We leverage the binary state variables to compute a deterministic upper bound for the optimal cost-to-go with provable convergence. In addition, we produce tight Lagrangian cuts by reusing past Lagrangian multipliers, which speeds up the computation of SND significantly. Thirdly, leveraging the optimal dynamic policy produced by our DRO model, we construct two easy-to-implement policies that alleviate the demand of solving large-scale optimization formulations in real-time deployment. Finally, we demonstrate the performance of our model and solution approaches through a realistic transmission grid in California and its wildfire data. We compare the solutions and policies obtained by our proposed multi-stage framework to a two-stage benchmark in terms of line switching decisions and their performance on reducing load shedding.

The remainder of the paper is organized as follows. Section \ref{sec:LitRev} reviews relevant research and positions our work in the context of recent advancements. Section \ref{sec:StochOpt} presents the multi-stage DRO model with DDU. Section \ref{sec:DDR-SND} extends the SND algorithm to solve the DRO model. Section \ref{sec:NumStudies} presents the numerical results, before we discuss and conclude in Section~\ref{sec:Discussion}. All technical proofs are relegated to the appendix and additional numerical results are reported in an online repository~\citep{DLSData}.

\section{Literature Review}
\label{sec:LitRev}
With its long history of study, power system optimization problems (such as unit commitment and line switching) continue being challenging to solve, especially when considering large-scale systems and stochastic elements~\citep{mohseni2022stochastic}. In particular, transmission grids are vulnerable to failures caused by random events such as natural disasters and wildfires~\citep{sayarshad2023evaluating}. As a response, various planning or operational strategies, such as distributed generation~\citep{mohagheghi2015optimal}, line switching~\citep{fisher2008optimal}, and de-energizing~\citep{yang2024multiperiod}, have been proposed to enhance the grid resiliency. In this context, two-stage models have been widely applied, e.g., with first-stage decisions pertaining to line switching and network configuration before any failures take place and second-stage decisions pertaining to post-failure energy re-dispatch and load shedding~\citep[see, e.g.,][]{nguyen2020preparatory,mohseni2022stochastic,jalilpoor2022network}.

The performance of stochastic optimization approaches can suffer when uncertainty cannot be accurately modeled. To remedy this, two-stage distributionally robust extensions of these approaches have been proposed to increase the degree of robustness \citep[see, e.g.,][]{zhang2017two,moreira2024distribution,pianco2024decision}. In contrast to these works, we consider a multi-stage model to adjust the grid topology dynamically, based on the most up-to-date wildfire propagation states.

In Table~\ref{table:LitRev}, we position our work with respect to the relevant state-of-the-art. To the best of our knowledge, multi-stage models that hedge against wildfire risks have been scarce to date. Within this literature, our work is most related to~\citep{yang2024multiperiod,yang2024multistage}, which considered stochastic optimization models for de-energizing grid components and achieving grid resiliency. Notably,~\cite{yang2024multiperiod,yang2024multistage} considered ``endogenous'' fires, which originate from random component failures but are independent of power system operations. For example, in their setting, a de-energized component can still fail, ignite, and propagate a fire. This renders their endogenous fires de facto decision-\emph{in}dependent uncertainty (DIU). In addition,~\citep{zou2018multistage} considered a multi-stage stochastic optimization model with DIU for dynamic unit commitment. In contrast to~\cite{yang2024multiperiod,yang2024multistage,zou2018multistage}, our model adopts DRO for modeling the line failure DDU and it also allows full flexibility for topology reconfiguration (i.e., both opening and closing lines).
\begin{table}[ht]
\centering
\caption{Comparison of relevant works in literature.}
\label{table:LitRev}
\small
\begin{tabular}{llccccc}
\hline
Work & Application & Uncertainty & Type & Horizon & Objective \\
\hline
\citet{mohagheghi2015optimal} & Distributed generation & Failure & DIU & Two-stage & RO \\
\citet{zou2018multistage} & Unit commitment & Demand & DIU & Multi-stage & SO \\
\citet{yang2024multiperiod} & Component de-energizing & Failure & DIU & Two-stage & SO \\
\citet{yang2024multistage} & Component de-energizing & Failure & DIU & Multi-stage & SO \\
\citet{nguyen2020preparatory} & Distributed generation & Failure & DIU & Two-stage & SO \\
\citet{hosseini2020computationally} & Line switching & -- & -- & Single-stage & DET \\
\textbf{This work} & Line switching & Failure & DDU & Multi-stage & DRO \\
\hline
\end{tabular}
\vspace{0.5em}
\small
DIU: Decision-independent uncertainty, DET: Deterministic, RO: Robust, SO: Stochastic, DRO: Distributionally robust optimization.
\end{table}

Extending the seminal work on stochastic dual dynamic programming~\citep{pereira1991multi},~\citet{zou2019stochastic} proposed the stochastic dual dynamic integer programming (SDDiP) algorithm to obtain optimal mixed-integer dynamic policies and applied it to multi-stage unit commitment with demand uncertainty~\citep{zou2018multistage}. Most applications of SDDiP assume stage-wise independent uncertainties for more efficient computation. In contrast, as we model an evolving wildfire process that samples the next wildfire state based on the current state, our model waives this assumption and addresses a more general, stage-wise \emph{dependent} scenario tree. Accordingly, we adopt the SND algorithm~\citep{zou2019stochastic}, which generalizes SDDiP.

The convergence of the SND algorithm relies on (i) statistical upper bounds of the optimal value and (ii) cutting planes that iteratively refine the lower approximation of the cost-to-go function in each stage. For (i), statistical upper bounds in~\cite{zou2019stochastic} become inapplicable for our DRO model with DDU. As an alternative, we propose deterministic upper bounds, enabling an accurate evaluation of the optimality gap. To improve (ii), existing works have proposed various approaches to generate stronger Lagrangian cuts for binary state variables (see~\cite{yang2024multistage,chen2022generating} and the references therein). We leverage this body of work to derive cuts for general (mixed-integer) state variables. In addition, we propose an algorithm to reuse Lagrangian multipliers in past SND iterations and generate tight Lagrangian cuts faster.

\section{Distributionally Robust Dynamic Line Switching}\label{sec:StochOpt}

We formulate a model to prescribe a dynamic line switching policy for a transmission grid. We describe the scenario tree in Section~\ref{subsec:AssumptNotation} and present the model in Section~\ref{subsec:MSDRO}. We derive a deterministic representation for this model and characterize the worst-case failure distribution in Section~\ref{subsec:DetRep}.
\begin{figure}
\centering
\includegraphics[width=0.35\textwidth]{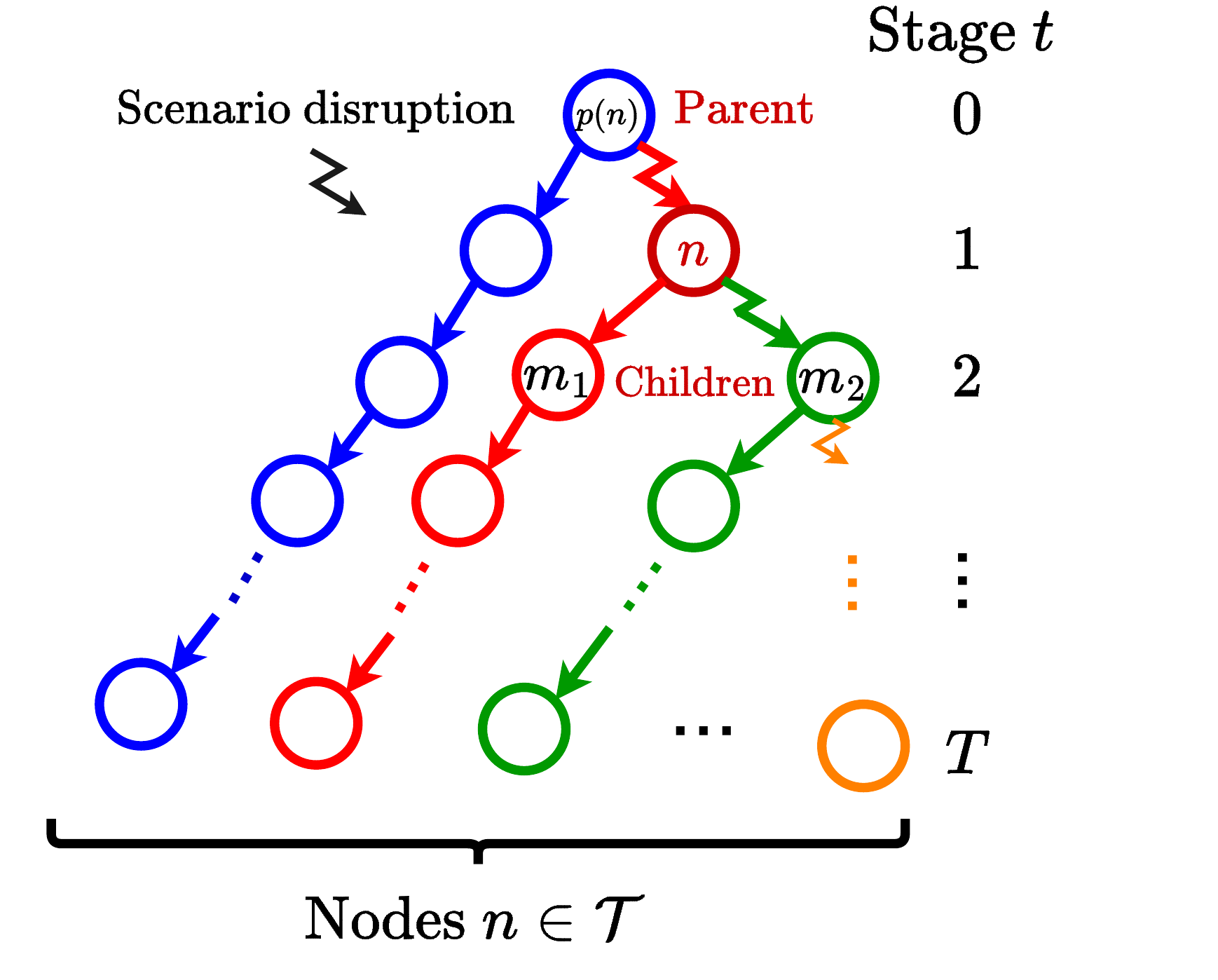}\caption{Scenario tree $\mathcal{T}$ of a general discrete stochastic process.
\label{fig:ScenTree}}
\end{figure}
\subsection{Scenario tree}
\label{subsec:AssumptNotation}
We model the wildfire propagation using a scenario tree \(\mathcal{T}\), whose nodes are organized in \(T\) levels/stages \(\mathcal{S}_t\) such that node \(1\) denotes the root node, \(\mathcal{T} = \bigcup_{t=1}^T \mathcal{S}_t\), \(\mathcal{S}_r\cap\mathcal{S}_t = \emptyset\) whenever \(r\neq t\), and each node \(n \in \mathcal{S}_t\) represents a possible wildfire state in time unit \(t\) (see Fig.~\ref{fig:ScenTree}). In addition, each non-leaf node \(n \in \mathcal{T}\) branches out to a set \(\mathcal{C}(n)\) of children nodes, and for each \(m \in \mathcal{C}(n)\), \(p_{nm}\) represents the (conditional) probability of branching to node \(m\) from \(n\). Conversely, for each non-root node \(n \in \mathcal{T}\), we denote by \(p(n)\) its parent node, i.e., \(n \in \mathcal{C}(p(n))\). By convention, we define \(p(1) := 0\) and \(\mathcal{C}(n) := \emptyset\) for all leaf nodes \(n\), and we call the path of wildfire state realizations from the root node to a leaf node \(n\) a ``scenario'' and denote it by \(\Pi(n)\).

The scenario tree \(\mathcal{T}\) is sufficiently general for modeling any finite and discrete stochastic process and for approximating a continuous one. For example, to model the wildfire propagation in a region, we can divide the region into a finite set \(\Sigma\) of cells and denote the state of each cell \(\sigma \in \Sigma\) in time unit \(t\) using an integer \(w^{\text{fire}}_{\sigma t}\), with \(w^{\text{fire}}_{\sigma t} = 0, 1, 2\) representing the cell being unburned, burning, and burned, respectively. This way, the state of the wildfire across all cells in time \(t\) can be encoded by a vector \(w^{\text{fire}}_t := \{w^{\text{fire}}_{\sigma t}: \forall \sigma \in \Sigma\}\) and \(\mathcal{T}\) can model the wildfire propagation by appropriately designating \(p_{nm}\) based on \(\mathbb{P}(w^{\text{fire}}_{t+1}|w^{\text{fire}}_t)\). Other data-driven approaches have also been proposed to construct scenario trees~\citep{heitsch2009scenario,oliveira2015time,yang2024multistage}. In this paper, we generate scenario trees based on historical wildfire perimeter data (see Appendix~\ref{apdx:ScenTree} for details).
\subsection{A multi-stage DRO model with DDU}
\label{subsec:MSDRO}
We start by following~\cite{hedman2010co} to formulate a deterministic line switching model for every node $n \in \mathcal{T}$. For ease of exposition, we summarize the nomenclature for all sets, parameters, and decision variables in Table~\ref{table:notation} and the detailed model in formulation~\eqref{model:S-OPF} of Appendix~\ref{apdx:notation}. This model takes in as parameters the current line availability states \(\tilde{\mathbf{a}}_n\), plus the line switching decisions \(\mathbf{z}_{p(n)}\), the power flows \((\mathbf{o}_{p(n)}, \mathbf{f}_{p(n)})\), and the power generation amounts \(\mathbf{p}_{p(n)}\) in the parent node \(p(n)\). Then, it seeks to minimize the generator production cost plus the load shedding cost, subject to phase angle limits, flow balance for each bus, and transmission line and generator capacity. 
We define a vector $\mathbf{x}_{n} := [\mathbf{z}_n,\mathbf{o}_n,\mathbf{f}_{n},\mathbf{p}_{n}]$ to group variables representing switching decisions, power flows, and power generation amounts; and we define another vector $\mathbf{y}_{n} := [\bm{\theta}_{n},\bm{\Delta}_{n}]$ to group variables representing voltage angles and load shedding amounts. Vectors \(\mathbf{x}_n\) and \(\mathbf{y}_n\) are different in that \(\mathbf{x}_n\) (\emph{inter}-stage variables) can be passed from node \(n\) to its children, while \(\mathbf{y}_n\) (\emph{intra}-stage variables) only affects the operations within node \(n\). Using these abstract notations, we express the deterministic line switching model in a compact form:\begin{equation*}
\min_{\mathbf{x}_{n},\mathbf{y}_{n}} \left\{f_{n}(\mathbf{x}_{n},\mathbf{y}_{n},\Tilde{\mathbf{a}}_{n}):A_n \mathbf{x}_{n} + W_{n}\mathbf{y}_{n} + C_{n} \mathbf{x}_{p(n)} + D_{n}\Tilde{\mathbf{a}}_{n} \geq h_{n}\right\},
\end{equation*}
where $f_{n}$ represents the objective function, and matrices $A_n,W_n,C_n,D_n$ and right-hand side vector $h_n$ encode the coefficients of constraints (see Appendix~\ref{apdx:notation} for nomenclature in Table~\ref{table:notation} and a detailed formulation in~\eqref{model:S-OPF}).

Then, we propose a multi-stage DRO model to prescribe an optimal line switching policy under line failure DDU. We consider two sources of uncertainties: firstly, the scenario tree \(\mathcal{T}\) models the (decision-independent) wildfire propagation uncertainty; and secondly, we model the (decision-dependent) line failure uncertainty through an ambiguity set, which depends on the latest wildfire state and the power flow on the line. Accordingly, the model is a dynamic program that, at each node \(n \in \mathcal{T}\), seeks to balance the immediate cost \(f_{n}(\mathbf{x}_{n},\mathbf{y}_{n},\Tilde{\mathbf{a}}_{n})\)and the expected cost-to-go with respect to the wildfire propagation distribution and the worst-case distribution of line failures. Specifically, the model evaluates \(\mathcal{Q}_1(\mathbf{x}_0, \mathbf{1})\) with a pre-specified system state \(\mathbf{x}_0\) before the planning horizon starts and zero line failure (\(\tilde{\mathbf{a}}_1 = \mathbf{1}\) almost surely), where
\begin{subequations}
\label{model:Q_n}
\begin{align}
\mathcal{Q}_{n}(\mathbf{x}_{p(n)},\Tilde{\mathbf{a}}_{n}) \ := \ \min_{\mathbf{x}_{n},\mathbf{y}_{n}} \quad & f_{n}(\mathbf{x}_{n},\mathbf{y}_{n},\Tilde{\mathbf{a}}_{n}) + \sum_{m \in \mathcal{C}(n)} p_{nm} \sup_{\mathbb{P} \in \mathcal{P}_m(\mathbf{x}_n)} \mathbb{E}_{\mathbb{P}}\left[ \mathcal{Q}_{m}(\mathbf{x}_{n},\Tilde{\mathbf{a}}_{m})\right] \label{model:Q_n-obj} \\
\text{s.t.}\quad & A_n \mathbf{x}_{n} + W_{n}\mathbf{y}_{n} + C_{n} \mathbf{x}_{p(n)} + D_{n}\Tilde{\mathbf{a}}_{n} \geq h_{n} \label{model:Q_n-con}
\end{align}
\end{subequations}
for all nodes \(n \in \mathcal{T}\). We model the line failure DDU using a DRO model, as opposed to a classical stochastic program, to take advantage of existing parametric prediction of line failure probability as a function of power flow magnitude (see, e.g.,~\cite{muhs2020characterizing,pianco2024decision}) and to waive the need to estimate such probability for each possible magnitude \(\mathbf{f}_n\). In particular, we follow~\cite{muhs2020characterizing,pianco2024decision} to define a moment ambiguity set
\begin{equation*}
\mathcal{P}_m(\mathbf{x}_{n}) = \left\{ 
\mathbb{P} \in \mathcal{P}(\mathcal{A}_{m}): \mathbb{E}_{\mathbb{P}}\left[ \tilde{\mathbf{a}}_{m}\right] \leq \beta_{m}\mathbf{x}_{n} + \gamma_{m} \right \}, 
\end{equation*}
\noindent where the line survival probability is bounded by an affine function of \(\mathbf{x}_n\) with constant matrix $\beta_{m} \in \mathbb{R}_{-}^{|\mathcal{L}| \times \text{dim}(\mathbf{x}_n)}$ and vector $\gamma_{m} \in \mathbb{R}_{+}^{|\mathcal{L}|}$, where \(\mathcal{L}\) denotes the set of transmission lines. Note that \((\beta_m, \gamma_m)\) depends on \(m\) and hence can be calibrated based on the latest wildfire state in node \(m\).
In addition, although the definition of \(\mathcal{P}_m(\mathbf{x}_{n})\) is general and allows the line survival probability to depend on all components of \(\mathbf{x}_n\), in this paper we focus on its (negative) correlation with the power flow magnitude $|f_{ln}|$ and the indicator \(o_{ln}\) for \(f_{ln}\) exceeding the nominal range. Finally, \(\mathcal{P}(\mathcal{A}_{m})\) denotes the set of all distributions supported on \(\mathcal{A}_{m}\), which is defined through
\begin{equation*}
    \mathcal{A}_{m} := \left\{ \Tilde{\mathbf{a}}_{m} \in \{0,1\}^{|\mathcal{L}|}: \; e^{\top}\Tilde{\mathbf{a}}_{m} \geq |\mathcal{L}| - K, \; \tilde{\mathbf{a}}_{m} \leq \tilde{\mathbf{a}}_{p(m)} \right\}. 
\end{equation*}
In other words, we allow for at most $K$ line failures and assume that a failed line during wildfires will remain dysfunctional throughout the planning horizon.

\subsection{Deterministic representation and worst-case distribution}
\label{subsec:DetRep}
Formulation~\eqref{model:Q_n} is challenging to solve directly because the worst-case expectation \(\displaystyle \sup_{\mathbb{P} \in \mathcal{P}_m(\mathbf{x}_n)} \mathbb{E}_{\mathbb{P}}\left[ \mathcal{Q}_{m}(\mathbf{x}_{n},\Tilde{\mathbf{a}}_{m})\right]\) embeds an infinite-dimensional optimization problem with respect to the probability distribution \(\mathbb{P}\). We derive a deterministic and finite-dimensional representation.
\begin{proposition} \label{prop:exp-rep}
For any fixed \(m \in \mathcal{T}\) and \(\mathbf{x}_n\), it holds that
\begin{subequations}
    \label{model:Q_min}
    \begin{align}
         \sup_{\mathbb{P} \in \mathcal{P}_m(\mathbf{x}_n)} \mathbb{E}_{\mathbb{P}}\left[ \mathcal{Q}_{m}(\mathbf{x}_{n},\Tilde{\mathbf{a}}_{m})\right] \ = \ \min_{\bm{\psi}_{m} \geq 0,\phi_{m}} \quad & \bm{\psi}_{m}^{\top} \beta_{m}\mathbf{x}_{n} + \bm{\psi}_{m}^{\top}\gamma_{m} + \phi_{m} \\
         \text{s.t.}\quad & \phi_{m} \geq \mathcal{Q}_{m}(\mathbf{x}_{n},\mathbf{a}_{m}) - \mathbf{a}_{m}^{\top} \bm{\psi}_{m}, \quad \forall \mathbf{a}_{m} \in \mathcal{A}_{m}. \quad \label{cst:Q_min:lowerapprox}
    \end{align}
\end{subequations}
\end{proposition}
The deterministic representation~\eqref{model:Q_min} is a linear program with respect to \((\bm{\psi}_{m}, \phi_m)\). In addition, it produces a bilinear program when integrated back to formulation~\eqref{model:Q_n} (due to the product term \(\bm{\psi}_{m}^{\top} \beta_{m}\mathbf{x}_{n}\)), which can be solved by commercial solvers (e.g., GUROBI) directly. On the other hand, constraints~\eqref{cst:Q_min:lowerapprox} are exponential in number, which is computationally prohibitive if \(K\) is large. Nevertheless, we take advantage of the structure of \(\mathcal{Q}_{m}(\mathbf{x}_{n},\tilde{\mathbf{a}}_{m})\) to efficiently separate these constraints. Indeed, as shown in Section~\ref{sec:DDR-SND}, the SND algorithm replaces \(\mathcal{Q}_{m}(\mathbf{x}_{n},\tilde{\mathbf{a}}_{m})\) with an (iteratively improved) lower approximation \(\underline{\mathcal{Q}}_m(\mathbf{x}_{n},\tilde{\mathbf{a}}_{m})\), which is the pointwise maximum of a set \(\mathcal{H}_m\) of affine functions
\begin{equation}
\max_{h \in \mathcal{H}_m} \left\{ (\pi^{h})^{\top}\mathbf{x}_{n} + (\tau^{h})^{\top}\tilde{\mathbf{a}}_m + \omega^{h} \right\}, \label{Q-rep}
\end{equation}
where \((\pi^{h}, \tau^{h}, \omega^{h})\) are constants to be generated in the process of the SND algorithm.
\begin{algorithm}
\small
    \caption{Separation of constraints~\eqref{cst:Q_min:lowerapprox}}\label{alg:separation}
    \begin{algorithmic}[1]
    \For{\(h = 1, \ldots, |\mathcal{H}_m|\)}
        \If{\(\mathbf{1}^{\top}\tilde{\mathbf{a}}_{p(m)} = |\mathcal{L}| - K\)}
        \State Set \(\mathbf{a}^h = \tilde{\mathbf{a}}_{p(m)}\);
        \Else
        \State Set \(\mathbf{a}^h = \tilde{\mathbf{a}}_{p(m)}\);
        \State Find a permutation of set \(\Upsilon := \{j: (\tau^h - \bm\psi_m)_j \leq 0\}\), denoted by \((1), \ldots, (|\Upsilon|)\), such that \((\tau^h - \bm\psi_m)_{(1)} \leq (\tau^h - \bm\psi_m)_{(2)} \leq \cdots \leq (\tau^h - \bm\psi_m)_{(\Upsilon)}\);
        \State Set \(\mathbf{a}^h_{(j)} = 0\) for all \(j = 1, \ldots, \min\{|\Upsilon|, \mathbf{1}^{\top}\tilde{\mathbf{a}}_{p(m)} - |\mathcal{L}| + K\}\);
        \EndIf
        \State Set \(v^h := (\pi^{h})^{\top}\mathbf{x}_{n} + (\tau^{h} - \bm{\psi}_m)^{\top}\mathbf{a}^h + \omega^{h}\);
    \EndFor
    \State Return the \(\mathbf{a}^h\) such that \(v^h \geq v^k\) for all \(k \in \mathcal{H}_m\);
    \end{algorithmic}
\end{algorithm}
The next proposition shows that the separation of constraints~\eqref{cst:Q_min:lowerapprox} is polynomial in \(|\mathcal{H}_m|\).
\begin{proposition} \label{prop:separation}
If \(\mathcal{Q}_{m}(\mathbf{x}_{n},\mathbf{a}_{m})\) admits a representation in the form of~\eqref{Q-rep}, then Algorithm~\ref{alg:separation} solves the separation problem of constraints~\eqref{cst:Q_min:lowerapprox}, i.e.,
\begin{align}
    \max_{\mathbf{a} \in \mathcal{A}_{m}}\{\mathcal{Q}_{m}(\mathbf{x}_{n},\mathbf{a}) - \mathbf{a}^{\top} \bm{\psi}_{m}\} \label{eq:MaxCut} \tag{SEP}
\end{align}
in time \(O(|\mathcal{H}_m||\mathcal{L}|\log(|\mathcal{L}|))\).
\end{proposition}
By Proposition~\ref{prop:separation}, we can quickly detect if a given solution \((\hat{\bm\psi}_m, \hat{\phi}_m)\) violates any of constraints~\eqref{cst:Q_min:lowerapprox}. Indeed, we can run Algorithm~\ref{alg:separation} to obtain an \(\mathbf{a}^h\) and then check if \(\hat{\phi}_m < \mathcal{Q}_m(\mathbf{x}_n, \mathbf{a}^h) - (\mathbf{a}^h)^{\top}\bm\hat{\bm\psi}_m\) (violation of~\eqref{cst:Q_min:lowerapprox} with respect to \(\mathbf{a}^h\)) or otherwise certify that \((\hat{\bm\psi}_m, \hat{\phi}_m)\) satisfies~\eqref{cst:Q_min:lowerapprox} with respect to all \(\mathbf{a} \in \mathcal{A}_m\). Accordingly, we can solve formulation~\eqref{model:Q_n} using delayed constraint generation. Specifically, we first solve a relaxation of~\eqref{model:Q_n} that ignores constraints~\eqref{cst:Q_min:lowerapprox} to obtain an incumbent solution \((\hat{\mathbf{x}}_n, \hat{\mathbf{y}}_n, \hat{\bm\psi}_m, \hat{\phi}_m)\). Then, we iteratively detect the violation of any of these constraints using Algorithm~\ref{alg:separation} and incorporate all violated constraints back into the relaxation till no violation can be found. At that point, the latest incumbent is then an optimal solution to~\eqref{model:Q_min}. Let set \(\mathcal{A}^*_m\) collect all \(\mathbf{a}^h\) obtained from Algorithm~\ref{alg:separation} before an optimal solution is certified. The following proposition derives a worst-case probability distribution for \(\tilde{\mathbf{a}}_m\) that attains \(\displaystyle \sup_{\mathbb{P} \in \mathcal{P}_m(\mathbf{x}_n)} \mathbb{E}_{\mathbb{P}}\left[ \mathcal{Q}_{m}(\mathbf{x}_{n},\Tilde{\mathbf{a}}_{m})\right]\).
\begin{proposition} \label{prop:worst-dist}
For any fixed \(m \in \mathcal{T}\) and \(\mathbf{x}_n\), it holds that
\begin{subequations}
    \label{model:Q_min-reduced}
    \begin{align}
         \sup_{\mathbb{P} \in \mathcal{P}_m(\mathbf{x}_n)} \mathbb{E}_{\mathbb{P}}\left[ \mathcal{Q}_{m}(\mathbf{x}_{n},\Tilde{\mathbf{a}}_{m})\right] \ = \ \min_{\bm{\psi}_{m} \geq 0,\phi_{m}} \quad & \bm{\psi}_{m}^{\top} \beta_{m}\mathbf{x}_{n} + \bm{\psi}_{m}^{\top}\gamma_{m} + \phi_{m} \label{wc-note-1} \\
         \text{s.t.}\quad & \phi_{m} \geq \mathcal{Q}_{m}(\mathbf{x}_{n},\mathbf{a}_{m}) - \mathbf{a}_{m}^{\top} \bm{\psi}_{m}, \quad \forall \mathbf{a}_{m} \in \mathcal{A}^*_{m}. \quad \label{cst:Q_min-reduced:lowerapprox}
    \end{align}
\end{subequations}
In addition, let \(\lambda^h\) denote the dual optimal solution associated with constraints~\eqref{cst:Q_min-reduced:lowerapprox} and define a probability distribution \(\mathbb{P}^*\) with \(\mathbb{P}^*\{\tilde{\mathbf{a}}_m = \mathbf{a}^h\} = \lambda^h\) for all \(h \in [|\mathcal{A}^*_m|]\). Then, it holds that
\[
\sup_{\mathbb{P} \in \mathcal{P}_m(\mathbf{x}_n)} \mathbb{E}_{\mathbb{P}}\left[ \mathcal{Q}_{m}(\mathbf{x}_{n},\Tilde{\mathbf{a}}_{m})\right] \ = \ \mathbb{E}_{\mathbb{P}^*}\left[ \mathcal{Q}_{m}(\mathbf{x}_{n},\Tilde{\mathbf{a}}_{m})\right].
\]
\end{proposition}
\section{A Stochastic Nested Decomposition (SND) Algorithm}
\label{sec:DDR-SND}
We adapt the SND algorithm to solve the multi-stage DRO model with DDU. We present the SND algorithm in Section~\ref{subsec:SND}. Furthermore, we provide various computing enhancement strategies, including deterministic upper bounds and faster generation of Lagrangian cuts in Section~\ref{subsec:UB}. 
\subsection{Algorithm design}
\label{subsec:SND}
We adapt the SND algorithm for multi-stage stochastic integer program~\citep{zou2019stochastic} to our DRO model. Our Algorithm~\ref{alg:SDN} maintains and iteratively refines a lower approximation \(\underline{\mathcal{Q}}_n(\mathbf{x}_{p(n)}, \tilde{\mathbf{a}}_n)\) of the cost-to-go function \(\mathcal{Q}_n(\mathbf{x}_{p(n)}, \tilde{\mathbf{a}}_n)\) through cutting planes. A main difference between~\cite{zou2019stochastic} and Algorithm~\ref{alg:SDN}, however, is that a part of our system state \(\tilde{\mathbf{a}}_n\) is DDU. Consequently, the transition of system state 
is \emph{decision-dependent}. Fortunately, we derived a worst-case distribution of \(\tilde{\mathbf{a}}_n\) in Proposition~\ref{prop:worst-dist} and can then sample from it for state transition.
\begin{algorithm}
\small\caption{SND algorithm for the DRO model with DDU}\label{alg:SDN}
    \begin{algorithmic}[1]
    \State Initialization: LB $\gets -\infty$, $i \gets 1$;
    \While{(Stopping Criterion Not Satisfied)}
    \State Sample $M$ scenarios $\Omega^{i}=\{\omega^{i}_{1},\dots,\omega^{i}_{M}\}$;
    \State \textbf{[Forward pass]}
    \For{$k=1,\dots,M$}
        \For{$n \in \omega^{i}_{k}$}
            \State Solve formulation~\eqref{model:Q_n} with $\mathcal{Q}_{m}(\cdot, \cdot)$ replaced by $\underline{\mathcal{Q}}_{m}(\cdot, \cdot)$ and store solution $(\hat{\mathbf{x}}_{n}^{i},\hat{\mathbf{y}}_{n}^{i})$;
            \State Sample $\hat{\mathbf{a}}_{m}^i$ from the worst-case distribution $\mathbb{P}^{*}$;
            \EndFor
        \EndFor

        \State \textbf{[Backward pass]}
    \For{$t = T-1,\dots,1$}
        \For{$n \in \mathcal{S}_{t}$} \label{alg-snd-note-1}
            \If{$n \in \omega^{i}_{k}$ for some $k \in \{1,\dots,M\}$}
                \For{$m \in \mathcal{C}(n)$}
                    \State Add cuts to strengthen the lower approximation $\underline{\mathcal{Q}}_{m}(\cdot, \cdot)$;
                \EndFor
            \EndIf
        \EndFor \label{alg-snd-note-2}
    \EndFor
    \State \textbf{[Lower bound]}
    \State Evaluate \(\mathcal{Q}_1(\mathbf{x}_0, \mathbf{1})\) using formulation~\eqref{model:Q_n} with \(n=1\) and $\mathcal{Q}_{m}(\cdot, \cdot)$ replaced by $\underline{\mathcal{Q}}_{m}(\cdot, \cdot)$;
    \State LB \(\gets \mathcal{Q}_1(\mathbf{x}_0, \mathbf{1})\) and \(i \gets i+1\);
    \EndWhile
    \end{algorithmic}
\end{algorithm}

The deterministic representation of formulation~\eqref{model:Q_n}, by Proposition~\ref{prop:exp-rep}, is a bilinear program. Although commercial solvers (e.g., GUROBI) can solve bilinear programs directly, they solve MILPs more efficiently by far. In addition, lower approximations of \(\mathcal{Q}(\mathbf{x}_n, \tilde{\mathbf{a}}_m)\) by cutting planes are computationally easier to obtain with state variables \(\mathbf{x}_n\) being binary, as opposed to being continuous or mixed-integer. As a result, in this paper we approximate all continuous state variables (power generation \(p_{gn}\) and power flow \(f_{ln}\)) using binary expansions~\citep{owen2002value}. 

Specifically, for each continuous state variable \(Z \in [L, U]\), we define auxiliary binary variables \(\{z_e \in \{0, 1\}, e \in [E]\}\) and approximate 
\begin{equation*}
    Z \approx L + s \sum_{e = 1}^{E} 2^{e-1} z_{e},
\end{equation*}
where $s$ is a pre-specified approximation precision (e.g., $s = 10^{-1},10^{-2}$) and $E:=\big\lfloor \log_{2}\left(\frac{U-L}{s}\right)\big\rfloor+1$. Under this approximation, \(\mathbf{x}_n\) becomes (purely) binary variables. Consequently, the bilinear term \(\bm{\psi}_{m}^{\top} \beta_{m}\mathbf{x}_{n}\) in the deterministic representation of formulation~\eqref{model:Q_n} can be linearized using standard McCormick inequalities and we can solve~\eqref{model:Q_n} as a MILP efficiently. In Theorem~\ref{thm:expansion} of Appendix~\ref{apdx:thm-expansion}, we show that the approximation error, in terms of the optimal value of our DRO model by applying the binary expansion, is \emph{linear} in precision \(s\). Consequently, for the rest of this section, we shall focus on solving the binary expansion approximation using a SND algorithm.

Throughout the algorithm, $\underline{\mathcal{Q}}_{n}(\mathbf{x}_{n},\tilde{\mathbf{a}}_{m})$ is characterized by the pointwise maximum of affine functions (see, e.g.,~\eqref{Q-rep}) and we iteratively update $\underline{\mathcal{Q}}_{n}(\mathbf{x}_{p(n)},\tilde{\mathbf{a}}_{n})$ by adding new cuts in the form
\begin{equation}
\underline{\mathcal{Q}}_{n}(\mathbf{x}_{p(n)},\tilde{\mathbf{a}}_{n}) \geq \pi^{\top}\mathbf{x}_{p(n)} + \tau^{\top}\tilde{\mathbf{a}}_n + \omega \label{eq:cut}
\end{equation}
to strengthen the lower approximation. In this paper, we follow~\cite{zou2019stochastic} and consider the following three families of cuts. To describe them, we rewrite~\eqref{model:Q_n} with respect to \(\underline{\mathcal{Q}}_{n}(\mathbf{x}_{p(n)},\Tilde{\mathbf{a}}_{n})\) as 
\begin{subequations}
\label{model:Q-underline}
\begin{align}
\underline{\mathcal{Q}}_{n}(\mathbf{x}_{p(n)},\Tilde{\mathbf{a}}_{n}) \ := \ \min_{\substack{\mathbf{x}_{n},\mathbf{y}_{n},\\ \mathbf{r}_{n}, \mathbf{w}_n: \text{\eqref{model:Q_n-con}}}} \quad & f_{n}(\mathbf{x}_{n},\mathbf{y}_{n},\mathbf{w}_{n}) + \sum_{m \in \mathcal{C}(n)} p_{nm} \sup_{\mathbb{P} \in \mathcal{P}_m(\mathbf{x}_n)} \mathbb{E}_{\mathbb{P}}\left[ \underline{\mathcal{Q}}_{m}(\mathbf{x}_{n},\Tilde{\mathbf{a}}_{m})\right] \nonumber\\
\text{s.t.} \quad & \ \mathbf{r}_{n} = \mathbf{x}_{p(n)}, \label{eq:fishing-x} \\
& \ \mathbf{w}_n = \tilde{\mathbf{a}}_n, \label{eq:fishing-a} \\
& \ \mathbf{r}_n \in \{0,1\}^{\text{dim}(\mathbf{x}_{p(n)})},\mathbf{w}_{n} \in \{0,1\}^{\text{dim}(\tilde{\mathbf{a}}_n)}, \label{binary-con}
\end{align}
\end{subequations}
where constraints~\eqref{eq:fishing-x}--\eqref{eq:fishing-a} make copies of the state variables \((\mathbf{x}_{p(n)}, \tilde{\mathbf{a}}_n)\). As both \(\mathbf{x}_{p(n)}\) and \(\tilde{\mathbf{a}}_n\) are binary and \(\underline{\mathcal{Q}}_{m}(\mathbf{x}_{n},\Tilde{\mathbf{a}}_{m})\) admits a piecewise linear representation as in~\eqref{Q-rep}, formulation~\eqref{model:Q-underline} is a mixed-binary linear program and strong duality holds when we relax constraints~\eqref{eq:fishing-x}--\eqref{eq:fishing-x} in the Lagrangian manner. More specifically, \(\underline{\mathcal{Q}}_{n}(\mathbf{x}_{p(n)},\Tilde{\mathbf{a}}_{n})\) equals the optimal value of the following (max-min) Lagrangian relaxation:
\begin{align}
& \max_{\pi, \tau}\min_{\substack{\mathbf{x}_n, \mathbf{y}_n,\\ \mathbf{r}_n,\mathbf{w}_n: \text{\eqref{model:Q_n-con},\eqref{binary-con}}}} \ f_{n}(\mathbf{x}_{n},\mathbf{y}_{n},\mathbf{w}_{n}) + \sum_{m \in \mathcal{C}(n)} p_{nm} \sup_{\mathbb{P} \in \mathcal{P}_m(\mathbf{x}_n)} \mathbb{E}_{\mathbb{P}}\left[ \underline{\mathcal{Q}}_{m}(\mathbf{x}_{n},\Tilde{\mathbf{a}}_{m})\right] - \pi^{\top}(\mathbf{r}_{n} - \mathbf{x}_{p(n)}) \nonumber \\
& \hspace{8em} - \tau^{\top}(\mathbf{w}_{n} - \tilde{\mathbf{a}}_n) \nonumber\\
= \ & \max_{\pi, \tau} \ \mathcal{R}_{\mathbf{x}_{p(n)}, \tilde{\mathbf{a}}_n}(\pi, \tau) \ := \ \Big\{\pi^{\top} \mathbf{x}_{p(n)} + \tau^{\top}\tilde{\mathbf{a}}_n + \mathcal{L}_n(\pi, \tau)\Big\}, \label{eq:Lag-max-min}
\end{align}
where \((\pi, \tau)\) are dual variables associated with~\eqref{eq:fishing-x}--\eqref{eq:fishing-a}, respectively, and
\begin{equation}\label{lag-func}
\resizebox{0.9\textwidth}{!}{
\(\displaystyle
\mathcal{L}_n(\pi, \tau) := \min_{\substack{\mathbf{x}_n, \mathbf{y}_n, \\\mathbf{r}_n,\mathbf{w}_n: \text{\eqref{model:Q_n-con}, \eqref{binary-con}}}} \ f_{n}(\mathbf{x}_{n},\mathbf{y}_{n},\mathbf{w}_{n}) + \sum_{m \in \mathcal{C}(n)} p_{nm} \sup_{\mathbb{P} \in \mathcal{P}_m(\mathbf{x}_n)} \mathbb{E}_{\mathbb{P}}\left[ \underline{\mathcal{Q}}_{m}(\mathbf{x}_{n},\Tilde{\mathbf{a}}_{m})\right] - \pi^{\top}\mathbf{r}_{n} - \tau^{\top}\mathbf{w}_{n}
\)}.
\end{equation}
As~\eqref{eq:Lag-max-min} is an unconstrained convex program for \((\pi, \tau)\), we can solve~\eqref{eq:Lag-max-min} by using, e.g., the subgradient algorithm. The strong duality of~\eqref{eq:Lag-max-min} gives rise to the validity and tightness of Lagrangian cuts.
\begin{proposition}[\textbf{Lagrangian Cuts}]
\label{prop:lag-cut}
At any binary $(\hat{\mathbf{x}}_{p(n)}, \hat{\mathbf{a}}_n)$, inequality~\eqref{eq:cut} holds with
\begin{align*}
(\pi, \tau) &\in \argmax_{\pi, \tau} \ \mathcal{R}_{\hat{\mathbf{x}}_{p(n)}, \hat{\mathbf{a}}_n}(\pi, \tau), \
\omega = \mathcal{L}_n(\pi, \tau).
\end{align*}
In addition, the cut is tight in the sense that
\begin{equation}
\underline{\mathcal{Q}}_{n}(\hat{\mathbf{x}}_{p(n)},\hat{\mathbf{a}}_{n}) =\pi^{\top}\hat{\mathbf{x}}_{p(n)} + \tau^{\top}\hat{\mathbf{a}}_n + \omega. \label{eq:tight-cut}
\end{equation}
\end{proposition}
We note that formulation~\eqref{eq:Lag-max-min} puts no restrictions on \((\pi, \tau)\) and so any \((\pi, \tau)\) produces a valid inequality. A convenient (but not necessarily optimal) choice of \((\pi, \tau)\) is the dual optimal solutions of the linear programming relaxation of~\eqref{model:Q-underline}, producing the strengthened Benders' cuts.
\begin{proposition}{({\bf Strengthened Benders' Cuts})}
For any binary \((\hat{\mathbf{x}}_{p(n)}, \hat{\mathbf{a}}_n)\), let \((\pi, \tau)\) represent dual optimal solutions associated with constraints~\eqref{eq:fishing-x}--\eqref{eq:fishing-a} to formulation~\eqref{model:Q-underline} with \(\mathbf{x}_{p(n)} = \hat{\mathbf{x}}_{p(n)}\), \(\tilde{\mathbf{a}}_n = \hat{\mathbf{a}}_n\), and no binary restrictions~\eqref{binary-con}. Then, inequality~\eqref{eq:cut} holds with \(\omega = \mathcal{L}_n(\pi, \tau)\).
\end{proposition}
A third class of cuts extend the integer optimality cuts for two-stage stochastic integer program~\citep{laporte1993integer}, using a global lower bound \(L_n\) of \(\mathcal{Q}_n(\mathbf{x}_{p(n)}, \tilde{\mathbf{a}}_n)\). In computation, \(L_n\) can be found, e.g., by solving the LP relaxation of formulation~\eqref{model:Q-underline}.
\begin{proposition}{({\bf Integer Optimality Cuts})} \label{prop:integer-cut}
For any binary \((\hat{\mathbf{x}}_{p(n)}, \hat{\mathbf{a}}_n)\) and lower bound \(L_n\) of \(\mathcal{Q}_n(\mathbf{x}_{p(n)}, \tilde{\mathbf{a}}_n)\), inequality~\eqref{eq:cut} holds with \(\pi = (\underline{\mathcal{Q}}_n(\hat{\mathbf{x}}_{p(n)}, \hat{\mathbf{a}}_n) - L_n)(2\hat{\mathbf{x}}_{p(n)} - \mathbf{1})\), \(\tau = (\underline{\mathcal{Q}}_n(\hat{\mathbf{x}}_{p(n)}, \hat{\mathbf{a}}_n) - L_n)(2\hat{\mathbf{a}}_n - \mathbf{1})\), and \(\omega = \underline{\mathcal{Q}}_n(\hat{\mathbf{x}}_{p(n)}, \hat{\mathbf{a}}_n) - (\underline{\mathcal{Q}}_n(\hat{\mathbf{x}}_{p(n)}, \hat{\mathbf{a}}_n) - L_n) (\mathbf{1}^{\top}\hat{\mathbf{x}}_{p(n)} + \mathbf{1}^{\top}\hat{\mathbf{a}}_n)\).
\end{proposition}
We close this section by confirming that Algorithm~\ref{alg:SDN} converges finitely to the global optimum of the DRO model \(\mathcal{Q}_1(\mathbf{x}_0, \mathbf{1})\), which follows from Theorem 2 of~\cite{zou2019stochastic} and the tightness of the Lagrangian cuts (see Proposition~\ref{prop:lag-cut}) and the integer optimality cuts (see Proposition~\ref{prop:integer-cut-tight}).
\begin{theorem}
Suppose that Algorithm~\ref{alg:SDN} adopts either Lagrangian cuts or integer optimality cuts to strengthen \(\underline{\mathcal{Q}}_m(\cdot, \cdot)\), then with probability one the forward pass solutions \(\{(\hat{\mathbf{x}}^i_n, \hat{\mathbf{y}}^i_n): n \in \mathcal{T}\}\) converge in a finite number of iterations to global optimal solutions to \(\mathcal{Q}_1(\mathbf{x}_0, \mathbf{1})\).
\end{theorem}
\subsection{Computational strengthening strategies}
\label{subsec:UB}
We propose the following three strategies to strengthen Algorithm~\ref{alg:SDN}.
\subsubsection{Deterministic upper bounds}
In Algorithm~\ref{alg:SDN}, a stopping criterion can be a threshold on either the number of iterations or the improvement in LB. However, neither criterion certifies global optimality (or quantifies suboptimality) of solving the DRO model. As an alternative, we propose a deterministic upper bound UB for \(\mathcal{Q}_1(\mathbf{x}_0, \mathbf{1})\). Then, we can terminate Algorithm~\ref{alg:SDN} when LB and UB are sufficiently close. To this end, similar to the lower approximations \(\underline{\mathcal{Q}}_{n}(\mathbf{x}_{p(n)},\Tilde{\mathbf{a}}_{n})\), we store an upper approximation \(\overline{\mathcal{Q}}_{n}(\mathbf{x}_{p(n)},\Tilde{\mathbf{a}}_{n})\) for each $\mathcal{Q}_{n}(\mathbf{x}_{p(n)},\Tilde{\mathbf{a}}_{n})$ and update it through
\begin{equation}
\overline{\mathcal{Q}}_{n}(\mathbf{x}_{p(n)},\Tilde{\mathbf{a}}_{n}) \ := \ \min_{\mathbf{x}_{n},\mathbf{y}_{n}: \text{\eqref{model:Q_n-con}}} \quad f_{n}(\mathbf{x}_{n},\mathbf{y}_{n},\Tilde{\mathbf{a}}_{n}) + \sum_{m \in \mathcal{C}(n)} p_{nm} \sup_{\mathbb{P} \in \mathcal{P}_m(\mathbf{x}_n)} \mathbb{E}_{\mathbb{P}}\left[ \overline{\mathcal{Q}}_{m}(\mathbf{x}_{n},\Tilde{\mathbf{a}}_{m})\right], \label{Q-overline}
\end{equation}
for all \(n \in \mathcal{T}\) and accordingly \(\text{UB} = \overline{\mathcal{Q}}_1(\mathbf{x}_0, \mathbf{1})\). We store each \(\overline{\mathcal{Q}}_{n}(\cdot, \cdot)\) as a boolean function: for the solutions \((\hat{\mathbf{x}}^i_{p(n)}, \hat{\mathbf{a}}^i_n)\) ever explored in forward passes of Algorithm~\ref{alg:SDN}, where \(i \in [I]\) is the iteration index and \(I\) is the total number of iterations by far, \(\overline{\mathcal{Q}}_{n}(\hat{\mathbf{x}}^i_{p(n)}, \hat{\mathbf{a}}^i_n)\) equals the latest optimal value \(\hat{u}^i_n\) of the above formulation; and otherwise it equals a large constant \(\mathcal{M} \geq \max\limits_{\mathbf{x}_{p(n)}, \tilde{\mathbf{a}}_n}\mathcal{Q}_n(\mathbf{x}_{p(n)}, \tilde{\mathbf{a}}_n)\) (we discuss the initialization of \(\mathcal{M}\) in Appendix~\ref{apdx:UpperBound}). We represent this function using an auxiliary binary variable \(\delta^i_{p(n)}\), which equals \(1\) if and only if \(\mathbf{x}_{p(n)} = \hat{\mathbf{x}}^i_{p(n)}\) for any \(i \in [I]\), or otherwise \(\delta^{I+1}_{p(n)} = 1\) if \(\mathbf{x}_{p(n)}\) differs from all \(\hat{\mathbf{x}}^i_{p(n)}\). Likewise, we define an auxiliary binary variable \(\gamma^i_n\), which equals \(1\) if and only if \(\tilde{\mathbf{a}}_{n} = \hat{\mathbf{a}}^i_{n}\), or otherwise \(\gamma^{I+1}_n = 1\). This gives rise to the following reformulation by Proposition~\ref{prop:exp-rep}: 
\begin{align*}
\sup_{\mathbb{P} \in \mathcal{P}_m(\mathbf{x}_n)} \mathbb{E}_{\mathbb{P}}\left[ \overline{\mathcal{Q}}_{m}(\mathbf{x}_{n},\Tilde{\mathbf{a}}_{m})\right] \ = \ \min_{\bm{\psi}_{m} \geq 0,\phi_{m}} \quad & \bm{\psi}_{m}^{\top} \beta_{m}\mathbf{x}_{n} + \bm{\psi}_{m}^{\top}\gamma_{m} + \phi_{m} \nonumber \\
         \text{s.t.}\quad & \phi_{m} \geq \sum_{i=1}^I \delta^i_n \gamma^i_m \hat{u}^i_m + \mathcal{M} (\delta^{I+1}_n + \gamma^{I+1}_m - \delta^{I+1}_n \gamma^{I+1}_m) \\
         & - \mathbf{a}_{m}^{\top} \bm{\psi}_{m}, \quad \forall (\mathbf{a}_{m}, \gamma_m) \in \overline{\mathcal{A}}_{m},
\end{align*}
where \(\overline{\mathcal{A}}_{m} := \{(\mathbf{a}_{m}, \gamma_m): \mathbf{a}_{m} \in \mathcal{A}_m, \sum_{i=1}^{I+1} \gamma^i_m = 1, \gamma^i_m = 1 \Leftrightarrow \mathbf{a}_m = \hat{\mathbf{a}}^i_m, \forall i \in [I]\}\). Similar to the representation~\eqref{model:Q_min}, the above formulation involves exponentially many constraints. Nevertheless, we can identify the violated constraints and incorporate them only when violated by solving the following MILP for fixed \((\mathbf{x}_n, \delta_n, \bm{\psi}_{m})\):
\begin{subequations}
\begin{align}
\max_{\mathbf{a}_m \in \mathcal{A}_m, \gamma_m \geq 0} \ & \ \sum_{i=1}^I \delta^i_n \gamma^i_m \hat{u}^i_m + \mathcal{M} (\delta^{I+1}_n + \gamma^{I+1}_m - \delta^{I+1}_n \gamma^{I+1}_m) - \mathbf{a}_{m}^{\top} \bm{\psi}_{m} \nonumber \\
\text{s.t.} \ & \ 1 - \|\mathbf{a}_m - \hat{\mathbf{a}}^i_m\|_1 \leq \gamma^i_m \leq 1 - \|\mathbf{a}_m - \hat{\mathbf{a}}^i_m\|_\infty, \quad \forall i \in [I], \label{Q-overline-con-1} \\
& \ \sum_{i=1}^{I+1} \gamma^i_m = 1, \label{Q-overline-con-2}
\end{align}
\end{subequations}
where constraints~\eqref{Q-overline-con-1}--\eqref{Q-overline-con-2} ensure \((\mathbf{a}_{m}, \gamma_m) \in \overline{\mathcal{A}}_{m}\). In lines~\ref{alg-snd-note-1}--\ref{alg-snd-note-2} of Algorithm~\ref{alg:SDN}, in addition to updating the lower approximation \(\underline{\mathcal{Q}}_m(\cdot, \cdot)\), we solve formulation~\eqref{Q-overline} with respect to \((\hat{\mathbf{x}}^i_{p(n)}, \hat{\mathbf{a}}_n)\) obtained from the latest forward pass and update \(\overline{\mathcal{Q}}_{n}(\cdot, \cdot)\) with the ensuing optimal value \(\hat{u}^i_n\). We close this section by confirming that this modification of Algorithm~\ref{alg:SDN} yields a series of upper bounds UB, which converge finitely to \(\mathcal{Q}_1(\mathbf{x}_0, \mathbf{1})\) from above.
\begin{theorem}
UB decreases and, with probability one, converges in a finite number of iterations to \(\mathcal{Q}_1(\mathbf{x}_0, \mathbf{1})\).
\end{theorem}
\subsubsection{Lagrangian cuts through continuous state variables}
\label{subsec:ContRelax}
We approximate the continuous state variables of \(\mathbf{x}_n\), including power flow \(f_{ln}\) and power generation \(p_{gn}\), through binary expansion. This generates tight Lagrangian cuts by Proposition~\ref{prop:lag-cut} but increases the dimension of \(\mathbf{x}_n\) as well as the state space. As an alternative, we can generate the Lagrangian cuts through the continuous state variables without applying the binary expansion approximation. Specifically, we generate the Lagrangian cuts by solving formulation~\eqref{eq:Lag-max-min}, but without the binary restrictions~\eqref{binary-con} for variables \((f_{ln}, p_{gn})\). This leads to weaker Lagrangian cuts because the Lagrangian relaxation of formulation~\eqref{model:Q-underline} pertaining to continuous \(\mathbf{x}_{p(n)}\) does not admit strong duality. Nevertheless, this decreases the dimension of \(\mathbf{x}_{p(n)}\) and can speed up the convergence of the lower approximation \(\underline{\mathcal{Q}}_n(\cdot, \cdot)\). In implementation of Algorithm~\ref{alg:SDN}, we can start by incorporating Lagrangian cuts through continuous \(\mathbf{x}_n\) and then switch to their binary expansion, e.g., when the Lagrangian cuts cease to improve LB. In Section~\ref{subsec:CutExps} and Table~\ref{table:ContinuousRepresentations}, we demonstrate the effectiveness of this strategy.
\subsubsection{Faster generation of Lagrangian cuts}
\label{subsec:ReusePi}
The Lagrangian cuts~\eqref{eq:cut} are valid and tight by Proposition~\ref{prop:lag-cut}. However, the search for the coefficients \((\pi, \tau)\) often involves a subgradient algorithm, which iteratively solves MILP~\eqref{lag-func} and is computationally heavy. In contrast, the integer optimality cuts specified in Proposition~\ref{prop:integer-cut} are generally considered weaker than the Lagrangian cuts due to the large magnitude of its cut coefficients. Nevertheless, these coefficients admit closed-form expressions and so generating integer optimality cuts is much faster. Our numerical experiments in Table~\ref{table:CutStudy} confirm these observations as Algorithm~\ref{alg:SDN} with integer optimality cuts (see I+SB), which incurs significantly more iterations, consistently converges faster than with Lagrangian cuts (see L+SB).

To seek a better trade-off between the strength and efficiency of generating Lagrangian cuts, we make two observations. First, although weaker in general, the integer optimality cut is in fact a \emph{tight Lagrangian cut} at any binary \((\hat{\mathbf{x}}_{p(n)}, \hat{\mathbf{a}}_n)\).
\begin{proposition} \label{prop:integer-cut-tight}
For any binary \((\hat{\mathbf{x}}_{p(n)}, \hat{\mathbf{a}}_n)\), let \(\pi = (\underline{\mathcal{Q}}_n(\hat{\mathbf{x}}_{p(n)}, \hat{\mathbf{a}}_n) - L_n)(2\hat{\mathbf{x}}_{p(n)} - \mathbf{1})\), \(\tau = (\underline{\mathcal{Q}}_n(\hat{\mathbf{x}}_{p(n)}, \hat{\mathbf{a}}_n) - L_n)(2\hat{\mathbf{a}}_n - \mathbf{1})\), and \(\omega = \underline{\mathcal{Q}}_n(\hat{\mathbf{x}}_{p(n)}, \hat{\mathbf{a}}_n) - (\underline{\mathcal{Q}}_n(\hat{\mathbf{x}}_{p(n)}, \hat{\mathbf{a}}_n) - L_n) (\mathbf{1}^{\top}\hat{\mathbf{x}}_{p(n)} + \mathbf{1}^{\top}\hat{\mathbf{a}}_n)\) as in the integer optimality cut (see Proposition~\ref{prop:integer-cut}). Then, equality~\eqref{eq:tight-cut} holds. In addition, \((\pi, \tau) \in \argmax_{\pi, \tau} \mathcal{R}_{\hat{\mathbf{x}}_{p(n)}, \hat{\mathbf{a}}_n}(\pi, \tau)\) and \(\omega = \mathcal{L}_n(\pi, \tau)\).
\end{proposition}
Proposition~\ref{prop:integer-cut-tight} suggests that, at any binary \((\hat{\mathbf{x}}_{p(n)}, \hat{\mathbf{a}}_n)\), there exist multiple tight Lagrangian cuts in the sense of~\eqref{eq:tight-cut} and the integer optimality cut is one of them, but with steep slopes.

Second, the objective function \(\mathcal{R}_{\mathbf{x}_{p(n)}, \tilde{\mathbf{a}}_n}(\cdot, \cdot)\) of formulations~\eqref{eq:Lag-max-min} remains nearly unchanged in consecutive iterations of Algorithm~\ref{alg:SDN}. Indeed, in each iteration Algorithm~\ref{alg:SDN} only updates the lower approximation \(\underline{\mathcal{Q}}_m(\cdot, \cdot)\) in formulation~\eqref{lag-func}, while the rest of~\eqref{lag-func} remains the same. This suggests that, when searching for the coefficients \((\pi, \tau)\) of a Lagrangian cut through the subgradient algorithm, Algorithm~\ref{alg:SDN} solves \emph{nearly the same} formulation~\eqref{eq:Lag-max-min} in consecutive iterations, and it will be a waste of effort if we start each subgradient algorithm from scratch. In contrast, we propose using the integer optimality cut coefficients as an ``anchor'' and to revise past maximizers for~\eqref{eq:Lag-max-min} to accelerate the subgradient algorithm. We make this concrete in Algorithm~\ref{alg:PiRecycle}. This algorithm starts from the most recent maximizer and checks its optimality to formulation~\eqref{eq:Lag-max-min} with respect to an updated (but similar) \(\mathcal{R}_{\mathbf{x}_{p(n)}, \tilde{\mathbf{a}}_n}(\cdot, \cdot)\). In case of optimal, we find a tight and strong Lagrangian cut in one shot; and if not optimal, we iteratively move towards the integer optimality cut coefficients to check other past maximizers. This procedure ends with a Lagrangian cut that is at least as strong as the integer optimality cut and we demonstrate its effectiveness in Section~\ref{subsec:CutExps} (see Table~\ref{table:CutStudy}).
\begin{algorithm}
\small
    \caption{Generation of Lagrangian cuts by reusing past cut coefficients}\label{alg:PiRecycle}
    \begin{algorithmic}[1]
    \State Input: iteration index \(i \gets 0\), iteration budget $I_{\text{max}}$, past cut coefficients $\Pi:=\{\pi^{h},\tau^{h}\}_{h=1}^{H}$, most recent cut coefficients \((\pi^H,\tau^H)\), step size \(\alpha\), neighborhood threshold \(\epsilon\), binary \((\hat{\mathbf{x}}_{p(n)}, \hat{\mathbf{a}}_n)\);
    \State Compute the integer optimality cut coefficients $(\pi_{\text{int}},\tau_{\text{int}})$ and set the iterate $(\bar{\pi},\bar{\tau}) \gets (\pi^{H},\tau^{H})$;
    \While{$i < I_{\text{max}}$}
    \If{$\mathcal{R}_{\hat{\mathbf{x}}_{p(n)}, \hat{\mathbf{a}}_n}(\bar{\pi},\bar{\tau}) = \mathcal{R}_{\hat{\mathbf{x}}_{p(n)}, \hat{\mathbf{a}}_n}(\pi_{\text{int}},\tau_{\text{int}})$}
    \State \textbf{break};
    \Else
    \State Consider the convex combination $(\pi,\tau) \gets \alpha (\pi_{\text{int}},\tau_{\text{int}}) + (1 - \alpha) (\bar{\pi},\bar{\tau})$;
    \State Find a $(\pi^{h^{*}},\tau^{h^{*}}) \in \argmin \{\|(\pi^{h},\tau^{h}) - (\pi,\tau)\|_{1}: \ (\pi^{h},\tau^{h}) \in \Pi\}$;
    \State \textbf{If} $\|(\pi^{h^{*}},\tau^{h^{*}}) - (\pi,\tau)\|_{1} \leq \epsilon$ \textbf{then}  $(\bar{\pi},\bar{\tau}) \gets (\pi^{h^{*}},\tau^{h^{*}})$ and remove \((\pi^{h^{*}},\tau^{h^{*}})\) from $\Pi$;
    \State \textbf{else}  $(\bar{\pi},\bar{\tau}) \gets (\pi,\tau)$;
    \State $i \gets i+1$;
    \EndIf
    \EndWhile
    \If{$\mathcal{R}_{\hat{\mathbf{x}}_{p(n)}, \hat{\mathbf{a}}_n}(\bar{\pi},\bar{\tau}) \neq \mathcal{R}_{\hat{\mathbf{x}}_{p(n)}, \hat{\mathbf{a}}_n}(\pi_{\text{int}},\tau_{\text{int}})$}
    \State $(\bar{\pi},\bar{\tau}) \gets (\pi_{\text{int}},\tau_{\text{int}})$;
    \EndIf
    \State Return $(\bar{\pi},\bar{\tau})$;
    \end{algorithmic}
\end{algorithm}
\section{Numerical Case Study}
\label{sec:NumStudies}
We present a numerical case study based on a California power grid. We introduce the case and experiment setup in Section \ref{subsec:Instance}, compare the effectiveness of different computational strategies in Section \ref{subsec:CutExps}, report the performance of our dynamic line switching model in Section \ref{subsec:CompExps}, and demonstrate the value of modeling DDU in Section \ref{subsec:Sensitivity}. Finally, in Section \ref{subsec:PolicyExtract} we construct two classes of easy-to-implement line switching policies and demonstrate their performance.
\begin{table}[ht]
\centering
\small
\caption{Characteristics of transmission grid}
\label{table:Characteristics}
\begin{tabular}{lrrrr}
\hline
Component & Characteristic & Min & Max & Average \\
\hline
Line & Rating (MW/h) & 20.1 & 110.4 & 73.2 \\
Bus & Load (MW/h) & 0.0 & 14.2 & 3.5 \\
Generator & Capacity (MW/h) & 2.3 & 120.6 & 88.2 \\
& Cost (\$/MWh) & 29.5 & 54.0 & 32.5 \\
Network & Total load (MW/h) & 170.3 & 220.2 & 192.2 \\
\hline
\end{tabular}
\end{table}
\subsection{Experiment setup}
\label{subsec:Instance}
We create a test instance based on a region in California and the transmission grid therein through the CATS dataset~\citep{taylor2023california}. In Table~\ref{table:Characteristics}, we report the ranges of the instance parameters. In addition, we generate a scenario tree for wildfire propagation based on historical wildfire perimeters~\citep{calfiredata} and detail the generation approach in Appendix~\ref{apdx:ScenTree}. Finally, we generate the bus load based on the CAISO 2019 load data~(\cite{caiso2019report}; see Figs. \ref{fig:Instance1}--\ref{fig:Instance2} in Appendix \ref{apdx:Instance} for a depiction of the instance data).

To calibrate the ambiguity set \(\mathcal{P}_m(\mathbf{x}_n)\), we set a base value $\overline{\gamma}$ based on the fuel data in the region as well as transmission line characteristics, and a base value $\bar{\beta}$ based on transmission line characteristics such as nominal capacity and thermal rating. We conduct sensitivity analysis for \(\gamma\) and \(\beta\) in Section~\ref{subsec:Sensitivity}. We ran all experiments on an Intel Xeon CPU with 12 cores @ 3.4 GHz and 128 GB of memory. We solved all MILP formulations using GUROBI 11 through Python 3.11.8. We generated the cutting planes in Algorithm~\ref{alg:SDN} using lazy-constraints in GUROBI.
\begin{table}[ht]
\centering
\small
\caption{Gap comparison between different representations of $\mathbf{x}$.}
\label{table:ContinuousRepresentations}
\begin{tabular}{lrrr}
\hline
Representation of $\mathbf{x}$ in Lagrangian cuts & Best gap (\%) & Iterations & Runtime (s) \\
\hline
Binary $\mathbf{f}$ and $\mathbf{p}$ & 74.3 & \textbf{24} & 391.2 \\
Continuous $\mathbf{f}$ + Binary $\mathbf{p}$ & 82.9 & 43 & \textbf{254.3} \\
Continuous $\mathbf{p}$ + Binary $\mathbf{f}$ & 60.3 & 35 & 301.2 \\
Proposed & \textbf{45.2} & 51 & 434.9 \\
\hline
\end{tabular}
\end{table}
\subsection{Comparisons of computing strategies}
\label{subsec:CutExps}
First, we evaluate the strategy of generating the Lagrangian cuts with respect to the continuous decision variables \((f_{ln}, p_{gp(n)})\), as opposed to their binary expansion approximations, as detailed in Section~\ref{subsec:ContRelax}. For comparison, we implemented the benchmark strategies of generating Lagrangian cuts based on (i) binary expansion of \((f_{ln}, p_{gp(n)})\), (ii) continuous \(f_{ln}\) and binary expansion of \(p_{gp(n)}\), and (iii) continuous \(p_{gp(n)}\) and binary expansion of \(f_{ln}\). We run Algorithm~\ref{alg:SDN} on an instance with \(T = 24\), \(|\mathcal{S}_{T}| = 150\) using these strategies until no improvement of LB for 5 consecutive iterations and then report the gap between the final LB and the optimal value in Table~\ref{table:ContinuousRepresentations}. From this table, we observe that the proposed strategy achieved a significantly smaller gap than other benchmarks in slightly longer runtime. Hence, we stick to this strategy for the rest of numerical case studies. 

\begin{table}[ht]
\centering
\small
\setlength{\tabcolsep}{3pt}
\caption{Comparison of time and iterations of different cut strategies to convergence}
\label{table:CutStudy}
\begin{tabular}{@{}rrrrrrrrrrrrrrrrr@{}}
\hline
&&\multicolumn{3}{c}{I}&\multicolumn{3}{c}{L}&\multicolumn{3}{c}{I+SB}&\multicolumn{3}{c}{L+SB}&\multicolumn{3}{c}{Algorithm \ref{alg:PiRecycle}+SB}\\
$T$ & $|\mathcal{S}_{T}|$&  1\%  & 0.1\% & \#Iter &  1\% & 0.1\% & \#Iter  &  1\% & 0.1\% & \#Iter  &  1\% & 0.1\% & \#Iter&  1\% & 0.1\% & \#Iter  \\
\hline
6 & 50 & 881 & 1446 & 252 & 909 & 1096 & 23 & 512 & 630 & 37 & 678 & 794 & \textbf{15} & {\bf 382} & {\bf 425} & 27 \\
6 & 100 & 1840 & 3141 & 378 & 2377 & 2983 & 45 & 1017 & 1302 & 64 & 1487 & 1813 & \textbf{26} & {\bf 894} & {\bf 1023} & 43 \\
6 & 150 & 3561 & 5737 & 593 & 2948 & 3492 & 45 & 2012 & 2431 & 92 & 3126 & 3597 & \textbf{47} & {\bf 1405} & {\bf 2231} & 69 \\
12 & 50 & 2230 & 3306 & 365 & 1468 & 1600 & \textbf{28} & 840 & 934 & 51 & 1956 & 2071 & 33 & {\bf 682} & {\bf 765} & 44 \\
12 & 100 & 2306 & 3512 & 339 & 2546 & 2852 & 48 & 1840 & 2102 & 83 & 2389 & 2599 & \textbf{36} & {\bf 1537} & {\bf 1729} & 58 \\
12 & 150 & 6924 & 9722 & 766 & 7452 & 7694 & \textbf{77} & 3895 & 4102 & 124 & 8102 & 8126 & 82 & {\bf 3402} & {\bf 3624} & 101 \\
24 & 50 & 2076 & 3304 & 458 & 2078 & 2431 & \textbf{37} & \textbf{840} & 1480 & 73 & 2617 & 2975 & 45 & 920 & \textbf{1252} & 81 \\
24 & 100 & 5491 & 7933 & 813 & 5999 & 6373 & 95 & 2957 & 3204 & 120 & 5041 & 5202 & \textbf{57}& {\bf 2498} & {\bf 2758} & 111 \\
24 & 150 & 11656 & 17596 & 1137 & 10567 & 11729 & 75 & \textbf{5203} & 7023 & 147 & 10079 & 10868 & \textbf{68} & 5810 & {\bf 6203} & 126 \\
\hline
\end{tabular}
\vspace{0.5em}
\small
{Columns 1\% and 0.1\% report the time in seconds to reach the corresponding optimality gap.}
\end{table}

Second, we evaluate the effectiveness of different cutting planes, as well as their combinations, for lower approximating $\mathcal{Q}_{m}(\mathbf{x}_{n},\mathbf{a}_{m})$. Specifically, we consider three types of cuts: integer optimality cuts (I), Lagrangian cuts (L), and strengthened Benders cuts (SB) in the backward pass of Algorithm~\ref{alg:SDN}. We tested (I) and (L) by themselves as they are tight cuts that can guarantee convergence. Additionally, we incorporated (SB) to accelerate convergence and tested the combinations (I+SB) and (L+SB). We report the runtime and the number of iterations for Algorithm~\ref{alg:SDN} converging to an optimality gap of 1\% and 0.1\%, respectively, in Table~\ref{table:CutStudy}. Putting the last strategy (Algorithm~\ref{alg:PiRecycle}+SB) aside, we observe that the per-iteration runtime for (I) is the shortest, yet more iterations are required. This suggests that, although the (I) cuts are cheap to generate, their strength is relatively weaker. In contrast, (L) achieved the same optimality gaps within much less number of iterations but similar or even longer runtime than (I). This indicates that Lagrangian cuts are stronger but take longer to obtain. It makes sense because producing each Lagrangian cut involves a subgradient algorithm and solving a series of MILP. In addition, incorporating (SB) significantly accelerated the convergence, with shorter runtime and less iterations. This indicates that (SB) cuts are good complements for both (I) and (L) cuts. In particular, the combination (I+SB) constantly outperforms other strategies on different test instances.

Third, we evaluate Algorithm~\ref{alg:PiRecycle}, which reuses dual multipliers from earlier iterations to generate Lagrangian cuts. We invoked this strategy every 5 iterations of Algorithm~\ref{alg:SDN} with $\epsilon=0.3$ and $I_{\text{max}} = 15$. From Table~\ref{table:CutStudy}, we observe that the proposed strategy significantly shortened the runtime. For example, it cuts the runtime of (L+SB) in half in most instances to achieve an optimality gap of 1\% or 0.1\%.  Even when compared to the best-performing strategy (I+SB) from the last comparison, Algorithm~\ref{alg:PiRecycle} was able to further shorten the runtime by around 15\%. In addition, this strategy produced more iterations than the combination (L+SB), suggesting that Algorithm~\ref{alg:PiRecycle} is indeed computationally cheaper than the subgradient algorithm.

\subsection{Value of dynamic line switching}
\label{subsec:CompExps}

We evaluate the value of adopting a dynamic line switching policy (MS), as opposed to an alternative two-stage line switching plan (TS) or simply no switching (NS). Here, TS refers to (preventively) reconfiguring the transmission grid at the root node of \(\mathcal{T}\) and then sticking to the grid topology throughout~\citep{yang2024multiperiod,hosseini2020computationally}, while other operations (power generation, phase angle, etc.) can still be dynamically adjusted (see Appendix~\ref{subsec:MultiScaleDecisions} for a detailed model). In addition, NS is a special case of TS without initial grid reconfiguration. 

In Fig.~\ref{fig:GapViz}, we compare the average load shedding costs and average operational costs of MS, TS, and NS across all 24 hours of simulation. {\color{black} We compute the average at stage $t$ by considering the costs of every node in the set $\mathcal{S}_{t}$ and weighing them by their probability of occurrence.} From Fig.~\ref{fig:LossGap}, we observe that TS and NS resulted in drastically higher load loss than MS. This demonstrates the value of adopting a dynamic line switching policy during wildfires. To demonstrate the flexibility of MS, we depict Fig.~\ref{fig:SolutionViz} and compare the switching decisions of MS and TS in two wildfire scenarios (denoted A and B). We observe that, while TS cannot adapt to these distinct wildfire propagation scenarios (it opened the same line), MS reconfigured the grid and rerouted the power flow accordingly. On the other hand, Fig.~\ref{fig:CostGap} shows that MS also incurred a higher operational cost than TS and NS (for extended comparisons, see the online repository~\cite{DLSData} of this paper).

\begin{figure}[ht]
\centering
\begin{subfigure}[b]{0.35\textwidth}
    \centering
    \includegraphics[width=\textwidth]{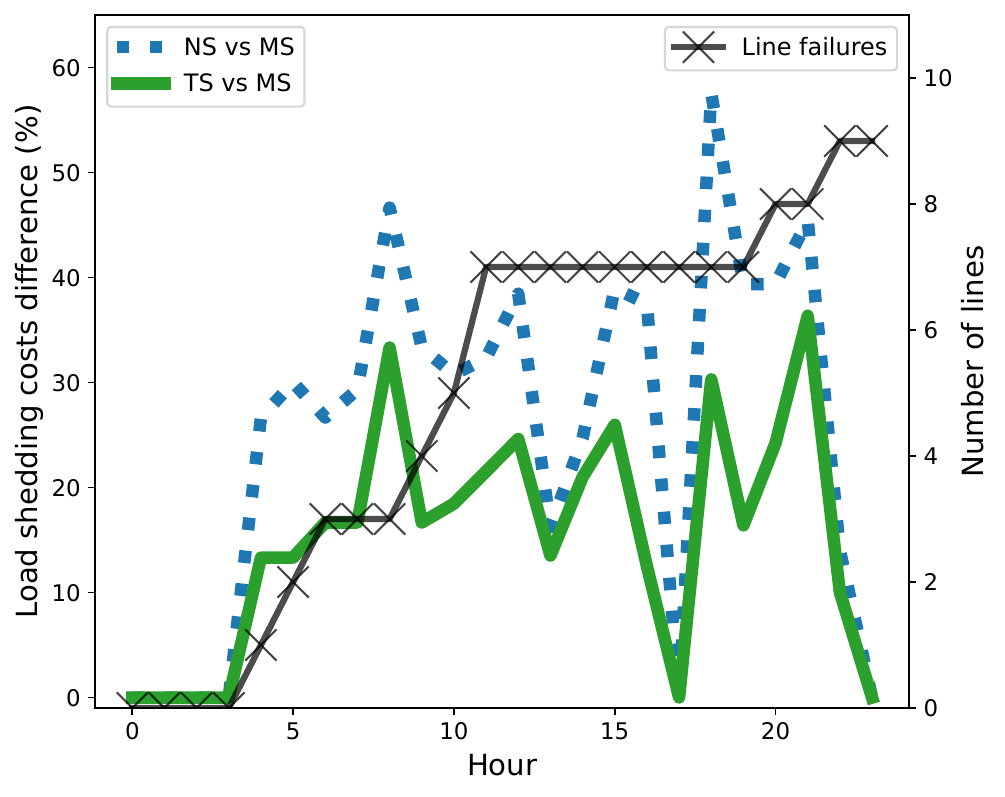}
    \caption{Gap of load loss}
    \label{fig:LossGap}
\end{subfigure}
\hspace{0.05\textwidth}
\begin{subfigure}[b]{0.35\textwidth}
    \centering
    \includegraphics[width=\textwidth]{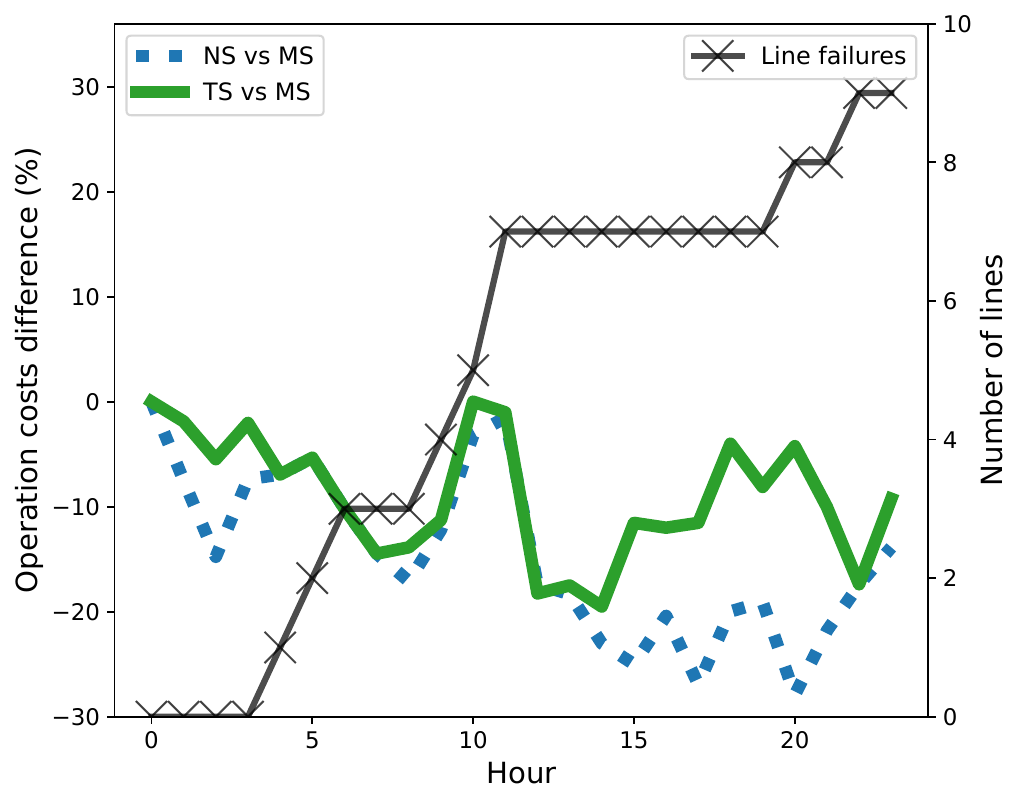}
    \caption{Gap of operational cost}
    \label{fig:CostGap}
\end{subfigure}
\caption{Average metric gaps between MS, TS, and NS}
\label{fig:GapViz}
\end{figure}

\begin{figure}[ht]
\centering
\begin{subfigure}[b]{0.25\textwidth}
    \centering
    \includegraphics[width=\textwidth]{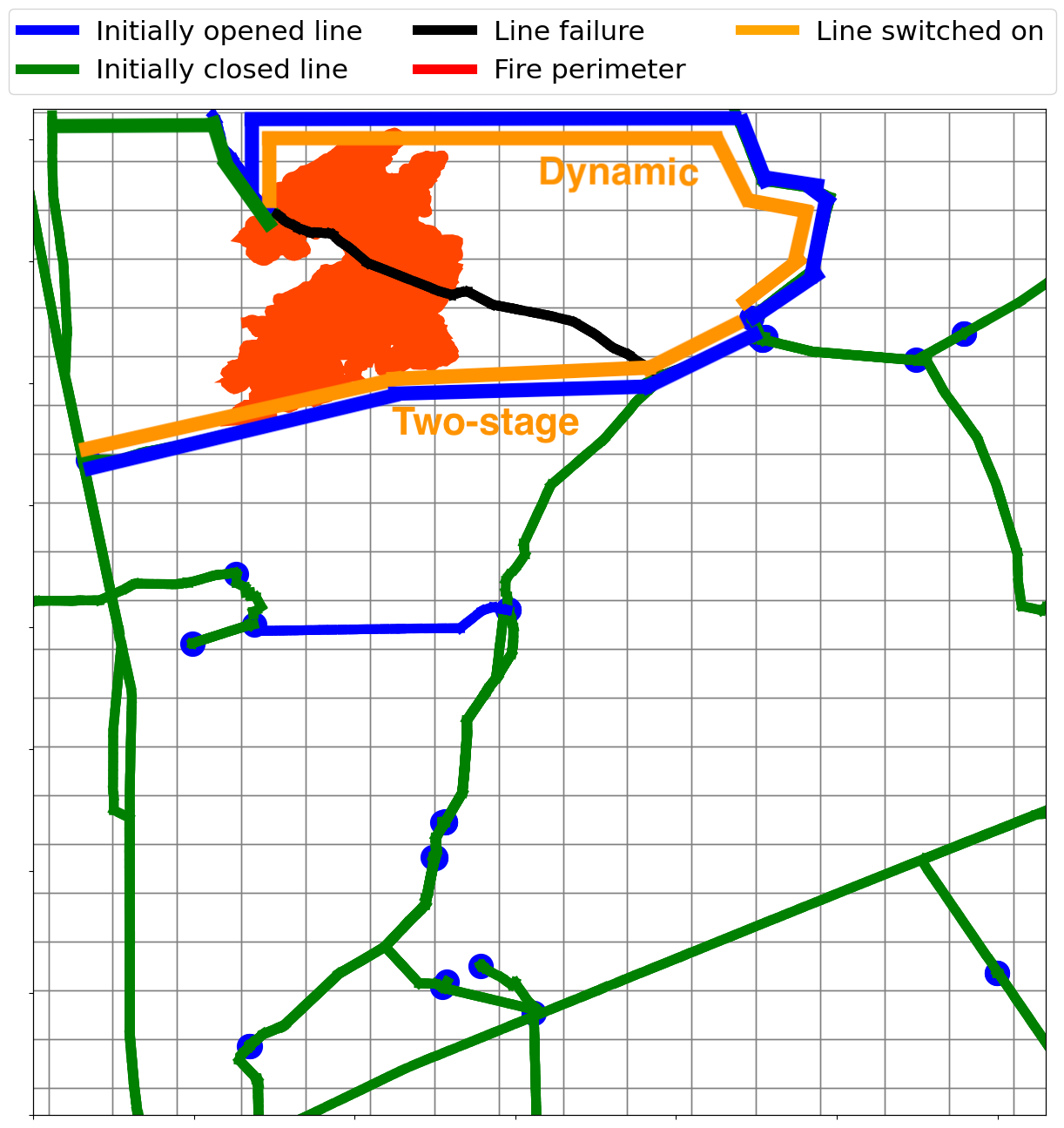}
    \caption{Switching in scenario A \label{fig:Switch1}}
\end{subfigure}
\hspace{0.05\textwidth}
\begin{subfigure}[b]{0.25\textwidth}
    \centering
    \includegraphics[width=\textwidth]{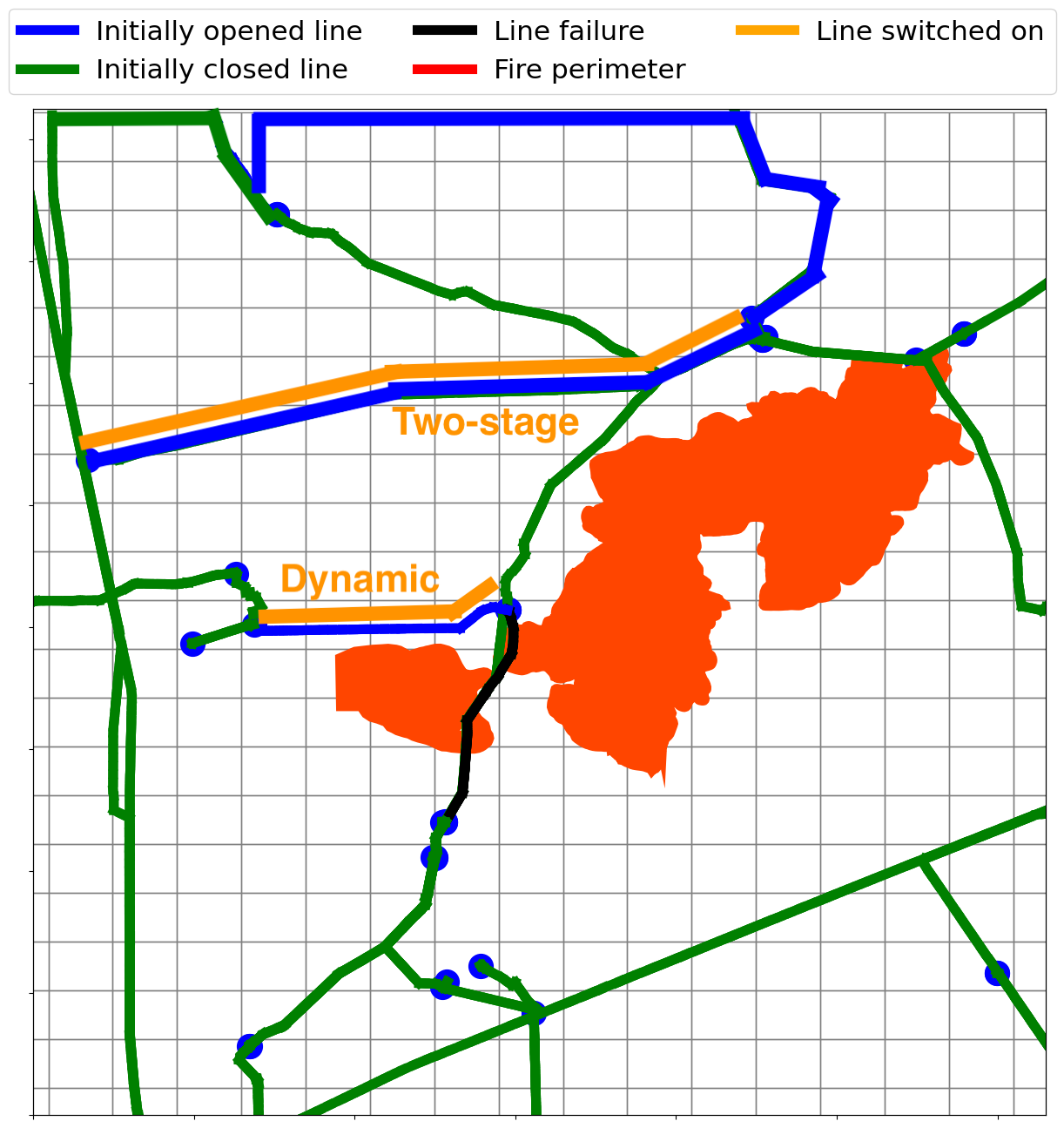}
    \caption{Switching in scenario B \label{fig:Switch2}}
\end{subfigure}
\caption{Comparison of switching decisions between MS and TS \label{fig:SolutionViz}}
\end{figure}

\subsection{Impact of DDU and DIU parameters}
\label{subsec:Sensitivity}
The ambiguity set \(\mathcal{P}_m(\mathbf{x}_n)\) models the line survival probability using a DDU component using parameter \(\beta\) and a DIU component using parameter \(\gamma\). To examine the impacts of DDU and DIU parameters, we visualize the optimal value of the MS model as a function of \(\beta\) and \(\gamma\) in Fig.~\ref{fig:SensSurfaces}. From Fig.~\ref{fig:SensSurfaces}, we observe that the impacts of the DDU parameter \(\beta\) are significantly larger than the DIU counterpart. This suggests that ignoring DDU can undermine the effectiveness of the MS model, demonstrating the value of modeling DDU in our study (for a complete set of sensitivity analyses, see the online repository~\cite{DLSData} of this paper).

\begin{figure}[ht]
\centering
\begin{subfigure}[b]{0.32\textwidth}
    \centering
    \includegraphics[width=\textwidth]{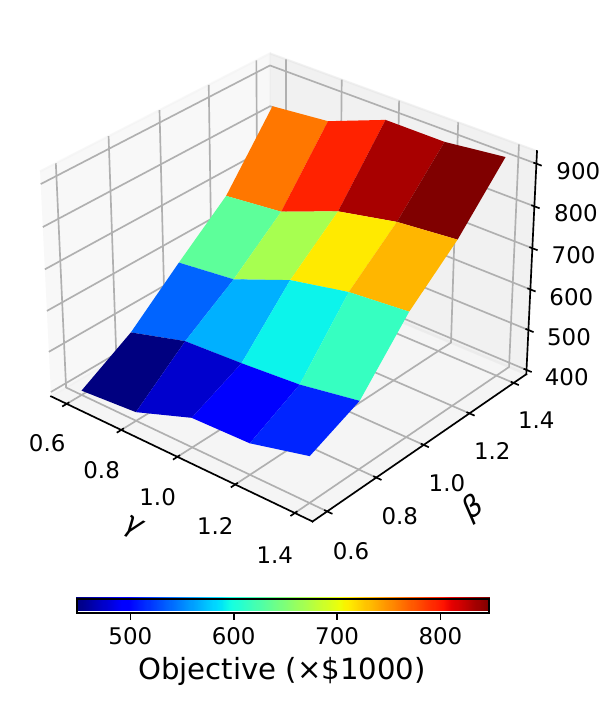}
    \caption{Left view \label{fig:Surf1}}
\end{subfigure}
\hspace{0.05\textwidth}
\begin{subfigure}[b]{0.3\textwidth}
    \centering
    \includegraphics[width=\textwidth]{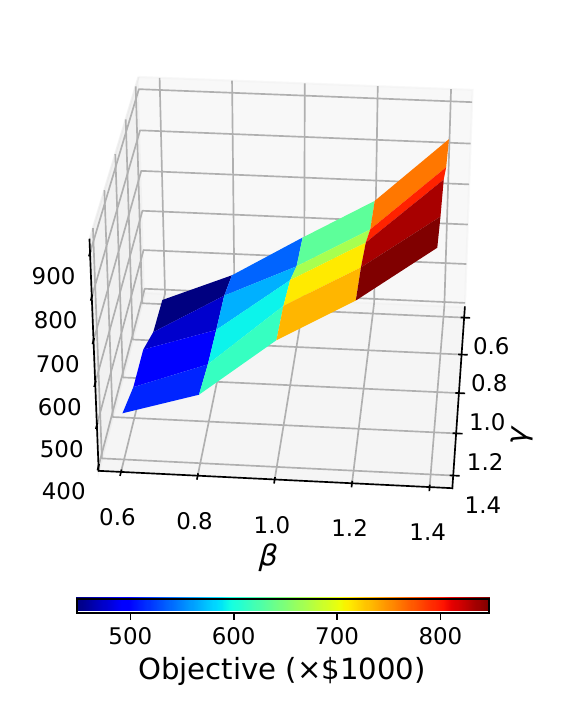}
    \caption{Right view \label{fig:Surf2}}
\end{subfigure}
\caption{Sensitivity of the MS optimal value in the DDU parameter \(\beta\) and DIU parameter \(\gamma\) \label{fig:SensSurfaces}}
\end{figure}

\subsection{Easy-to-implement policies}
\label{subsec:PolicyExtract}

We produce three line switching policies based on the final value function approximations \(\underline{\mathcal{Q}}_m(\cdot, \cdot)\) obtained from the SND algorithm. Given a time unit \(t\), the wildfire state \(w^{\text{fire}}_t\), and the grid state \((\mathbf{x}_{t-1}, \tilde{\mathbf{a}}_t)\), we first project \(w^{\text{fire}}_t\) to a node \(n \in \mathcal{S}_t\) encoding the closest wildfire state. Then, we produce the following three policies based on \(\underline{\mathcal{Q}}_m(\cdot, \cdot)\).
\begin{description}
    \item[{\bf Dynamic policy}] solves formulation~\eqref{model:Q_n} pertaining to node \(n\) and state \((\mathbf{x}_{t-1}, \tilde{\mathbf{a}}_t)\) with \(\mathcal{Q}_m(\cdot, \cdot)\) replaced by \(\underline{\mathcal{Q}}_m(\cdot, \cdot)\), and applies the ensuing solution \((\mathbf{x}_n, \mathbf{y}_n)\).
    \item[{\bf Topology policy}] is the same with the dynamic policy except that, when solving formulation~\eqref{model:Q_n}, it shrinks the feasible region of line switching to a set of grid topologies we collect from the final iteration of the SND algorithm. We detail the generation of this topology set in Algorithm~\ref{alg:Topology} in Appendix~\ref{apdx:Policies}. 
    \item[{\bf Mapping policy}] seeks to establish a deterministic look-up table that maps to the switching decision \(z_{lt}\) from the current state \(z_{l,t-1}\) of line \(l\), as well as the availability \(\tilde{\mathbf{a}}_t^{\mathcal{L}(l)}\) of the redundant lines in \(\mathcal{L}(l) := \{j \in \mathcal{L}: \text{line $l$ and $j$ feeds to the same bus}\}\). We sample the scenario tree in multiple replications and, in each replication, retrieve the decision for \(z_{lt}\) from the dynamic policy. Then, we round the average of these \(z_{lt}\) among all replications to either zero or one. We detail this approach in Algorithm~\ref{alg:Mapping} in Appendix~\ref{apdx:Policies}.
\end{description}

From a computing perspective, dynamic policy demands solving an MILP (e.g., formulation~\eqref{model:Q_n}) for each time unit \(t\) in its implementation. This is cheaper than solving the entire multi-stage DRO model (with a receding horizon, from time \(t\) to \(T\)), but can still be time-consuming with a large-scale transmission grid. As an alternative, the topology policy only chooses among a select set of topologies. We can implement it much more efficiently and interpret it more easily. Mapping policy completely waives the need of solving any optimization problems. We note that, although the (offline) preparation of the look-up tables demands resolving MILPs, the (online) implementation of these tables consumes no time. We evaluate the out-of-sample performance of these policies. To this end, we shall refer to the scenario tree we use in Algorithm~\ref{alg:SDN} as the training (in-sample) tree and perform out-of-sample simulations in the following two new trees: 

\begin{itemize}
    \item We regenerate a scenario tree using the same transition probabilities as in the training tree and refer to it as the out-of-sample tree.
    \item We generate another scenario tree with a more hazardous transition probabilities, with faster and/or larger propagation of the wildfire. We refer to it as the stress-test tree.
\end{itemize}

When implementing the above three policies, we project each new scenario (unobserved in the training tree) of the out-of-sample tree and stress-test tree to the closest scenario in the training tree with respect to 1-norm. We compare the performance of the topology policy and the mapping policy with the TS and the NS approaches in Table~\ref{table:TestComparison}, which reports the average and standard deviation of the performance gap (between these approaches and the dynamic policy) across 10 replications. From this table, we observe that the topology and mapping policies outperform both TS and NS significantly. This suggests that even the simplified versions of the dynamic policy retain an operational advantage of adaptability and outperform the non-adaptive policies (for additional results of this comparison, see the online repository~\cite{DLSData} of this paper).

\begin{table}[ht]
\centering
\small
\caption{Performance comparison of various policies in out-of-sample simulation.}
\label{table:TestComparison}
\begin{tabular}{ccrrr}
\hline
& & \multicolumn{3}{c}{Average (Standard Deviation)} \\
Approach vs Dynamic & Metric (\%) & Training & Out-of-sample & Stress-test \\
\hline
TS & Objective & 15.83 & 15.75 (1.03) & 24.58 (4.79) \\
& Operation & -9.54 & -9.97 (2.63) & -8.07 (4.53) \\
& Load shedding & 26.01 & 26.11 (3.29) & 37.77 (5.61) \\
NS & Objective & 30.13 & 30.33 (0.82) & 40.41 (4.22) \\
& Operation & -11.73 & -11.32 (1.84) & -10.15 (2.78) \\
& Load shedding & 47.92 & 47.11 (2.06) & 60.83 (5.67) \\
Topology & Objective & 5.78 & 5.14 (0.89) & 8.74 (3.62) \\
& Operation & -7.79 & -9.08 (0.82) & -7.57 (3.04) \\
& Load shedding & 11.23 & 10.88 (0.95) & 15.33 (2.40) \\
Mapping & Objective & 4.35 & 4.12 (1.23) & 10.47 (3.74) \\
& Operation & -0.35 & -1.38 (1.52) & -3.19 (2.17) \\
& Load shedding & 6.23 & 6.34 (2.01) & 16.02 (3.84) \\
\hline
\end{tabular}
\end{table}

\section{Conclusion}
\label{sec:Discussion}

We proposed a multi-stage DRO model with DDU to address the dynamic line switching of a transmission grid amidst wildfire propagation. We extended the SND algorithm proposed in~\cite{zou2018multistage} to solve this model and proposed strategies to enhance its computational performance, including deterministic upper bounds and faster generation of Lagrangian cuts. In addition, to facilitate online deployment of the proposed model, we proposed policies that are cheaper to implement and easier to interpret. We demonstrated the proposed model and computational strengthening strategies in a case study using a real-world transmission system and wildfire data.

\appendix
\section{Appendix to the paper}
\subsection{Nomenclature and formulation}
\label{apdx:notation}
We present a detailed formulation for the deterministic nodal line switching model as follows, whose nomenclature is summarized in Table~\ref{table:notation}.

\allowdisplaybreaks
\begin{subequations}
\label{model:S-OPF}
\begin{align} 
    \min_{\mathbf{x}_{n},\mathbf{y}_{n}} \quad& \sum_{g \in \mathcal{G}_{n}}C_{g}p_{gn} + \sum_{b \in \mathcal{N}} C^{l}\Delta_{bn}\label{obj:OPF}\\
     \mbox{s.t.}\quad& \overline{\theta}^{\min} \leq \theta_{bn} \leq \overline{\theta}^{\max}, \forall b \in \mathcal{N}, \label{cst:trans:angle}\\ 
     & \sum_{l \in l(b,\cdot)}f_{ln} - \sum_{l \in l(\cdot,b)}f_{ln} + \sum_{g \in g(b)} p_{gn} + \Delta_{bn} = d_{bn}, \forall b \in \mathcal{N},  \label{cst:trans:flowBalance} \\
     & \overline{F}_{l}^{\min} z_{ln} \leq f_{ln} \leq \overline{F}_{l}^{\max} z_{ln}, \forall l \in \mathcal{L}, \label{cst:trans:LineTransBoundz}\\
     & z_{ln} \leq a_{ln} , \forall l \in \mathcal{L}, \label{cst:trans:za}\\
     & B_{l}(\theta_{bn}-\theta_{b'n}) - f_{ln} + (2-z_{ln}-a_{ln})\mathcal{M}_{l} \geq 0, \forall l \in \mathcal{L}, (b,b') \in \mathcal{N},\label{cst:trans:OpenLine1}\\
     & B_{l}(\theta_{bn}-\theta_{b'n}) - f_{ln} - (2-z_{ln}-a_{ln})\mathcal{M}_{l} \leq 0, \forall l \in \mathcal{L}, (b,b') \in \mathcal{N},\label{cst:trans:OpenLine2}\\
     & P_{g}^{\min} \leq  p_{gn} \leq P_{g}^{\max}, \forall g \in \mathcal{G}, \label{cst:trans:GenBound}\\
     & p_{gn} - p_{gp(n)} \leq R_{g}^{+}, \forall g \in \mathcal{G},\label{cst:trans:Ramp1}\\
     & p_{gp(n)} - p_{gn} \leq R_{g}^{-}, \forall g \in \mathcal{G}, \label{cst:trans:Ramp2}\\
     & f_{ln} - (\overline{F}^{\text{max}}_{ln} - \overline{F}^{\text{nom}}_{ln})o_{ln} \leq  \overline{F}^{\text{nom}}_{ln} , \forall  l \in \mathcal{L}, \label{cst:trans:OverNomPos}\\
     & f_{ln} + (\overline{F}^{\text{max}}_{ln} - \overline{F}^{\text{nom}}_{ln})o_{ln} \geq  -\overline{F}^{\text{nom}}_{ln} , \forall  l \in \mathcal{L}, \label{cst:trans:OverNomNeg}\\
     &p_{gn} \geq 0, \forall g \in \mathcal{G},  \label{cst:trans:GenDom}\\
     &\Delta_{bn} \geq 0, \forall b \in \mathcal{N}, \label{cst:trans:AngleDom}\\
     &z_{ln},o_{ln} \in \{0,1\},\forall l \in \mathcal{L}. \label{cst:trans:SwitchDom}
\end{align}
\end{subequations}
In formulation~\eqref{model:S-OPF}, the objective function~\eqref{obj:OPF} seeks to minimize the generator production cost plus the load shedding cost, constraints~\eqref{cst:trans:angle} describe limits on the phase angles, constraints~\eqref{cst:trans:flowBalance} account for the flow balance for each bus, constraints~\eqref{cst:trans:LineTransBoundz} describe the transmission line capacity limits, depending on if the line is closed (\(z_{ln} = 1\)) or open (\(z_{ln} = 0\)), constraints~\eqref{cst:trans:za} make sure that we can close a line only if it is available (\(a_{ln} = 1\)) in the first place, constraints~\eqref{cst:trans:OpenLine1}--\eqref{cst:trans:OpenLine2} enforce the relationship between power flows and phase angles using DC power flow when the line is both available and closed, with \(\mathcal{M}_l := |B_{l}(\overline\theta^{\max} - \overline\theta^{\min})|\) denoting a sufficiently large constant, constraints~\eqref{cst:trans:GenBound} describe the generator minimum and maximum generation amounts, constraints~\eqref{cst:trans:Ramp1}--\eqref{cst:trans:Ramp2} describe the generator ramp-rate limits, and constraints~\eqref{cst:trans:OverNomPos}--\eqref{cst:trans:OverNomNeg} determine whether the power flow along each transmission line is within its nominal capacity (\(o_{ln}=0\)) or not (\(o_{ln}=1\)). 

\begin{table}
    \caption{A summary of sets, parameters, and variables.\label{table:notation}}\small
    \begin{tabular}{ll}
    \hline
    &\textbf{Sets}\\
    \hline
    \(\mathcal{C}(n)\) & Set of children nodes of node \(n\). \\
    $\mathcal{G}$ & Set of generators, indexed by $g$. \\
    $g(b)$ & Set of generators at bus $b \in \mathcal{N}$. \\
    $\mathcal{L}$ & Set of transmission lines, indexed by $l$. \\
    $l(b,\cdot)$ & Set of transmission lines with $b \in \mathcal{N}$ as the ``from'' bus. \\
    $l(\cdot,b)$ & Set of transmission lines with $b \in \mathcal{N}$ as the ``to'' bus. \\
    $\mathcal{N}$& Set of buses, indexed by $b$.\\
    \(\Pi(n)\) & Set of nodes in the scenario path leading to node \(n\).\\
    $\mathcal{S}_{t}$ & Set of nodes in stage $t$ of \(\mathcal{T}\).  \\
    $\mathcal{T}$ & Scenario tree for the wildfire propagation stochastic process. \\
    \hline
    &\textbf{Parameters}\\
    \hline
    $a_{ln}$ & Availability of transmission line $l \in \mathcal{L}$ at node $n \in \mathcal{T}$ (1 if available, 0 otherwise).\\
    $B_{l}$ & Susceptance of transmission line $l \in \mathcal{L}$.\\
    $\beta_{ln}$ & Sensitivity of the survival probability to the active power flow \(f_{ln}\) of line $l \in \mathcal{L}$ at node $n \in \mathcal{T}$.\\
    $C_{g}$ & Unit generation cost of generator $g \in \mathcal{G}$.\\
    $C^\text{L}_b$ & Unit cost of load shedding at bus \(b \in \mathcal{N}\).\\
    $d_{bn}$ & Real power load at bus $b \in \mathcal{N}$ for node $n \in \mathcal{T}$.\\
    $\overline{F}^{\max}_{l},\overline{F}^{\min}_{l}$ & Maximum/minimum ratings of transmission line $l \in \mathcal{L}$.\\
    $\overline{F}^{\text{nom}}_{l}$ & Nominal transmission rating of line $l \in \mathcal{L}$.\\
    $\gamma_{ln}$ & Nominal survival probability of line $l \in \mathcal{L}$ at node $n \in \mathcal{T}$.\\
    $K$ & Maximum number of line failures throughout the planning horizon. \\
    $P^{\max}_{g},P^{\min}_{g}$ & Maximum/minimum generation amounts of generator $g \in \mathcal{G}$.\\
    $p_{nm}$ & Transition probability from node $n$ to node $m \in \mathcal{C}(n)$ in the scenario tree $\mathcal{T}$. \\
    $T$ & Number of decision-making stages.\\
    $\overline{\theta}^{\max},\overline{\theta}^{\min}$ & Maximum/minimum bus voltage angle.\\
    \hline
    &\textbf{Decision variables}\\
    \hline
     $o_{ln}$ &  Indicator variable for \(f_{ln}\) exceeding the nominal rating $\overline{F}^{\text{nom}}_{l}$ of line \(l\) (1 if exceeding, 0 if not exceeding).\\
    $p_{gn}$ & Real power supply from generator $g \in \mathcal{G}$ at bus $b \in \mathcal{N}$ and node $n \in \mathcal{T}$.\\
    $f_{ln}$ & Real power flow
    on line $l \in \mathcal{L}$ at node $n \in \mathcal{T}$.   \\
    $\theta_{bn}$& Voltage angle at bus $b$ and node $n \in \mathcal{T}$.\\
    $\Delta_{bn}$ & Amount of load shedding at bus $b$ and node $n \in \mathcal{T}$. \\
    $\mathbf{x}_{n}$& = $[\mathbf{z}_n,\mathbf{o}_n,\mathbf{f}_{n},\mathbf{p}_{n}]$ Vector of inter-stage decision variables at node \(n \in \mathcal{T}\).\\
    $\mathbf{y}_{n}$& = $[\bm{\theta}_{n},\bm{\Delta}_{n}]$ Vector of intra-stage decision variables at node \(n \in \mathcal{T}\).\\
     $z_{ln}$ & Binary variable indicating a switching action of line $l$ at node $n \in \mathcal{T}$ (1 if closing, 0 if opening).\\
    \hline
\end{tabular}
\end{table}

\subsection{Generation of the scenario tree}
\label{apdx:ScenTree}

    We employ a data-driven approach to generate the scenario tree in the case study. We begin by discretizing the geographical region in our instance by considering a rectangular grid, such that we can have xx cells. We construct the scenario tree in a backwards fashion. We begin by having as many terminal nodes in the last stage, as we have historical fire perimeters in the region. Then, we consider a cellular automaton process, with basic propagation rules considering neighborhood fire propagation over the 24-hour period we simulate. We note that we do not consider dynamic wind conditions, nor more complex propagation models, but our framework can admit scenario trees that are generated by any other valid method. Then, with the backward-simulated scenarios, we compare the state of the cells starting from an initial state without fires, and merge scenarios as long as they have the same wildfire state, accumulating their probabilities, only branching whenever the fires change as they are simulated. In Figure \ref{fig:ScenarioGeneration}, we show the conceptual process of generating the scenario tree for the case study. Given that we assume a one-direction relationship, where wildfire impact line availability, we reduce the number of scenarios by only differentiating branching on them, i.e., differentiating between scenarios, if the wildfire impacts cells that are associated with lines that can fail, further reducing the scenario tree size.

    \begin{figure}
    \centering
    \includegraphics[width=0.5\textwidth]{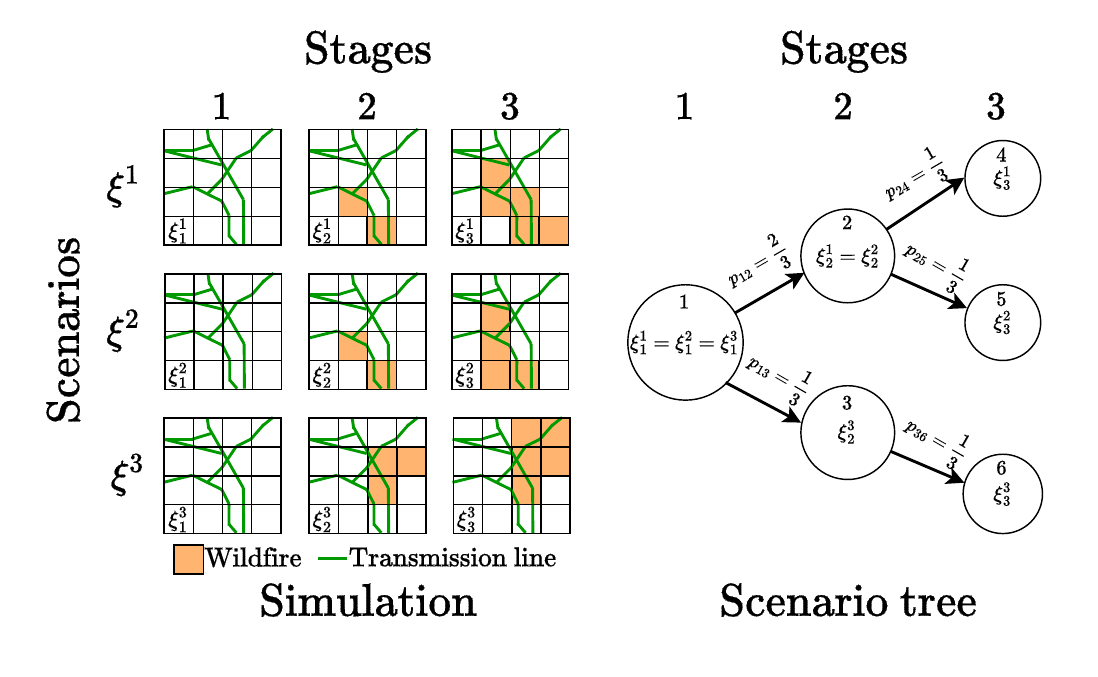}\caption{Scenario tree generation process\label{fig:ScenarioGeneration}}
    \end{figure}
    
\subsection{Approximation error of the binary expansion approximation}\label{apdx:BinExp}

\begin{theorem} \label{thm:expansion}
Assume that \(f_n(\mathbf{x}_n, \mathbf{y}_n, \tilde{\mathbf{a}}_n)\) is Lipschitz continuous in \((\mathbf{x}_n,\mathbf{y}_n)\), i.e.,
\[
|f_n(\mathbf{x}_n, \mathbf{y}_n, \tilde{\mathbf{a}}_n) - f_n(\mathbf{x}'_n, \mathbf{y}'_n, \tilde{\mathbf{a}}_n)| \leq L \|(\mathbf{x}_n,\mathbf{y}_n) - (\mathbf{x}'_n,\mathbf{y}'_n)\|_{\infty}, \; \forall \mathbf{x}_n, \mathbf{x}'_n, \mathbf{y}_n, \mathbf{y}'_n, \tilde{\mathbf{a}}_n, n \in \mathcal{T}
\]
for a constant \(L \geq 0\). Let \(z^* := \mathcal{Q}_1(\mathbf{x}_0, \mathbf{1})\) without the binary expansion approximation and \(z^*_{\text{bin}}\) represent the optimal value with the approximation. Then, it holds that \(z^* \leq z^*_{\text{bin}} \leq z^* + Ks\) for a constant \(K \geq 0\), where \(K\) depends only on \(L\) and the parameters for defining \(\mathcal{Q}_1(\mathbf{x}_0, \mathbf{1})\).
\end{theorem}

\begin{proof}[\textbf{Proof of Theorem~\ref{thm:expansion}}\label{apdx:thm-expansion}]
The binary expansion approximation on \(\mathbf{x}_n\) is equivalent to incorporating a new constraint \(\mathbf{x}_n \in S(\mathbf{x}_n) := \big\{z \in \mathbb{R}^{\text{dim}(\mathbf{x}_n)}: z_n \in \{L_n, L_n+s, \ldots, \min(U_n, L_n + (2^{E_n} - 1)s)\}, \forall n \in [N]\big\}\), where \(\bm{L}, \bm{U}\) are lower and upper bounds of \(\mathbf{x}_n\) and \(E_n = \big\lfloor \log_{2}\left(\frac{U_n-L_n}{s}\right)\big\rfloor+1\). Then, the binary expansion is a conservative approximation of the multi-stage DRO model and so \(z^* \leq z^*_{\text{bin}}\). In addition, we note that \(S(\mathbf{x}_n)\) is an \(s\)-net for the hypercube \([\bm{L}, \bm{U}]\). That is, for any \(\mathbf{x}_n \in [\bm{L}, \bm{U}]\), there exists a \(\bm{z} \in S(\mathbf{x}_n)\) such that \(\|\mathbf{x}_n - \bm{z}\|_\infty \leq s\).

To prove the second inequality, we use mathematical induction. More specifically, we will focus on a further conservative approximation that approximates both \(\mathbf{x}_n\) and \(\mathbf{y}_n\) through binary expansion. First ({\bf base case}), for each leaf node \(n\) and any fixed \(\mathbf{x}_{p(n)}\) and \(\Tilde{\mathbf{a}}_{n}\), recall that
\begin{align*}
\mathcal{Q}_{n}(\mathbf{x}_{p(n)},\Tilde{\mathbf{a}}_{n})= \min_{\mathbf{x}_{n},\mathbf{y}_{n}} \quad & f_{n}(\mathbf{x}_{n},\mathbf{y}_{n},\Tilde{\mathbf{a}}_{n})\\
\text{s.t.}\quad & A_n \mathbf{x}_{n} + W_{n}\mathbf{y}_{n} + C_{n} \mathbf{x}_{p(n)} + D_{n}\Tilde{\mathbf{a}}_{n} \geq h_{n}.
\end{align*}
We let \(\mathcal{Q}^{\mathbf{x}}_{n}(\mathbf{x}_{p(n)},\Tilde{\mathbf{a}}_{n})\) denote the optimal value of this formulation after adding the constraint \(\mathbf{x}_n \in S(\mathbf{x}_n)\). Now suppose that we apply the binary expansion on both \(\mathbf{x}_n\) and \(\mathbf{y}_n\), that is, add the constraints \(\mathbf{x}_n \in S(\mathbf{x}_n)\) and \(\mathbf{y}_n \in S(\mathbf{y}_n)\) to the above formulation and denote by \(\mathcal{Q}^{\mathbf{x,y}}_{n}(\mathbf{x}_{p(n)},\Tilde{\mathbf{a}}_{n})\) the ensuing optimal value. Then, \(\mathcal{Q}_{n}(\mathbf{x}_{p(n)},\Tilde{\mathbf{a}}_{n}) \leq \mathcal{Q}^{\mathbf{x}}_{n}(\mathbf{x}_{p(n)},\Tilde{\mathbf{a}}_{n}) \leq \mathcal{Q}^{\mathbf{x,y}}_{n}(\mathbf{x}_{p(n)},\Tilde{\mathbf{a}}_{n})\). Likewise, we have \(z^*_{\text{bin}} \equiv \mathcal{Q}^{\mathbf{x}}_{1}(\mathbf{x}_{0},\mathbf{1}) \leq \mathcal{Q}^{\mathbf{x,y}}_{1}(\mathbf{x}_{0},\mathbf{1})\).

But \(S(\mathbf{x}_n) \times S(\mathbf{y}_n)\) forms an \(s\)-net. Then, for an optimal solution \((\mathbf{x}^*_n, \mathbf{y}^*_n)\) to the above formulation, there exists an \((\overline{\mathbf{x}}_n, \overline{\mathbf{y}}_n) \in S(\mathbf{x}_n) \times S(\mathbf{y}_n)\) such that
\begin{equation*}
    \begin{cases}
        A_n \overline{\mathbf{x}}_{n} + W_{n}\overline{\mathbf{y}}_{n} + C_{n} \mathbf{x}_{p(n)} + D_{n}\Tilde{\mathbf{a}}_{n} \geq h_{n} \\[0.5em]
    \|(\overline{\mathbf{x}}_n, \overline{\mathbf{y}}_n) - (\mathbf{x}^*_n, \mathbf{y}^*_n)\|_{\infty} \leq s.
    \end{cases}
\end{equation*}
It follows that 
\begin{align*}
    \mathcal{Q}^{\mathbf{x,y}}_{n}(\mathbf{x}_{p(n)},\Tilde{\mathbf{a}}_{n}) \ \leq & \ f_n(\overline{\mathbf{x}}_{n}, \overline{\mathbf{y}}_{n}, \tilde{\mathbf{a}}_n) \\
\leq & \ f_n(\mathbf{x}^*_{n}, \mathbf{y}^*_{n}, \tilde{\mathbf{a}}_n) + L\|(\overline{\mathbf{x}}_n, \overline{\mathbf{y}}_n) - ({\mathbf{x}}_n, {\mathbf{y}}_n)\|_{\infty} \\
\leq & \ \mathcal{Q}_{n}(\mathbf{x}_{p(n)},\Tilde{\mathbf{a}}_{n}) + Ls,
\end{align*}
where the first inequality follows from the sub-optimality of \((\overline{\mathbf{x}}_n, \overline{\mathbf{y}}_n)\) and the second inequality follows from the Lipschitz continuity of \(f_n\). Furthermore, since \(f_n\) is linear and decision variables \((\mathbf{x}_n,\mathbf{y}_n)\) are binary in the formulation of \(\mathcal{Q}^{\mathbf{x,y}}_{n}(\mathbf{x}_{p(n)},\Tilde{\mathbf{a}}_{n})\), the proximity result of binary integer program (Theorem 2.2 in~\cite{blair1977value}) yields, for some \(E_n \geq 0\)
\begin{equation*}
    |\mathcal{Q}^{\mathbf{x,y}}_{n}(\mathbf{x}_{p(n)},\Tilde{\mathbf{a}}_{n}) - \mathcal{Q}^{\mathbf{x,y}}_{n}(\mathbf{x}'_{p(n)},\Tilde{\mathbf{a}}_{n})| \leq E_n\|\mathbf{x}_{p(n)} - \mathbf{x}'_{p(n)}\|_{\infty}, \; \forall \mathbf{x}_{p(n)}, \mathbf{x}'_{p(n)}.
\end{equation*}
Second ({\bf induction step}), for each node \(n \in \mathcal{S}_t\) with \(t \leq T-1\) and any fixed \(\mathbf{x}_{p(n)}\) and \(\Tilde{\mathbf{a}}_{n}\), recall that
\begin{align*}
\mathcal{Q}_{n}(\mathbf{x}_{p(n)},\Tilde{\mathbf{a}}_{n})= \min_{\mathbf{x}_{n},\mathbf{y}_{n}} \quad & f_{n}(\mathbf{x}_{n},\mathbf{y}_{n},\Tilde{\mathbf{a}}_{n}) + \sum_{m \in \mathcal{C}(n)} p_{nm} \sup_{\mathbb{P} \in \mathcal{P}_m(\mathbf{x}_n)} \mathbb{E}_{\mathbb{P}}\left[ \mathcal{Q}_{m}(\mathbf{x}_{n},\Tilde{\mathbf{a}}_{m})\right]\\
\text{s.t.}\quad & A_n \mathbf{x}_{n} + W_{n}\mathbf{y}_{n} + C_{n} \mathbf{x}_{p(n)} + D_{n}\Tilde{\mathbf{a}}_{n} \geq h_{n}
\end{align*}
and assume that, for all \(m \in \mathcal{C}(n)\), \(\mathbf{x}_n\), \(\mathbf{x}'_n\), and \(\Tilde{\mathbf{a}}_{m}\), 
\begin{enumerate}
\item \(\mathcal{Q}^{\mathbf{x,y}}_{m}(\mathbf{x}_{n},\Tilde{\mathbf{a}}_{m}) \leq \mathcal{Q}_{m}(\mathbf{x}_{n},\Tilde{\mathbf{a}}_{m}) + L_ms\) for some \(L_m \geq 0\);
\item \(|\mathcal{Q}^{\mathbf{x,y}}_{m}(\mathbf{x}_{n},\Tilde{\mathbf{a}}_{m}) - \mathcal{Q}^{\mathbf{x,y}}_{m}(\mathbf{x}'_{n},\Tilde{\mathbf{a}}_{m})| \leq E_m \|\mathbf{x}_{n} - \mathbf{x}'_{n}\|_\infty\) for some \(E_m \geq 0\).
\end{enumerate}
Then, for fixed \(\mathbf{x}_n\) and \(\mathbb{P} \in \mathcal{P}_m(\mathbf{x}_n)\), assumption (i) implies that
\begin{equation*}
    \mathbb{E}_{\mathbb{P}}\left[\mathcal{Q}^{\mathbf{x,y}}_{m}(\mathbf{x}_{n},\Tilde{\mathbf{a}}_{m})\right] \leq \mathbb{E}_{\mathbb{P}}\left[\mathcal{Q}_{m}(\mathbf{x}_{n},\Tilde{\mathbf{a}}_{m})\right] + L_ms.
\end{equation*}
Driving both sides to supremum with respect to \(\mathbb{P}\) and noting that \(\sum_{m \in \mathcal{C}(n)} p_{nm} \equiv 1\), we have
\begin{equation*}
    \sum_{m \in \mathcal{C}(n)} p_{nm} \sup_{\mathbb{P} \in \mathcal{P}_m(\mathbf{x}_n)} \mathbb{E}_{\mathbb{P}}\left[ \mathcal{Q}^{\mathbf{x,y}}_{m}(\mathbf{x}_{n},\Tilde{\mathbf{a}}_{m})\right] \leq \sum_{m \in \mathcal{C}(n)} p_{nm} \sup_{\mathbb{P} \in \mathcal{P}_m(\mathbf{x}_n)} \mathbb{E}_{\mathbb{P}}\left[ \mathcal{Q}_{m}(\mathbf{x}_{n},\Tilde{\mathbf{a}}_{m})\right] + \left(\sum_{m\in \mathcal{C}(n)}p_{nm}L_m\right)s.
\end{equation*}
In addition, assumption (ii) implies that, for any \(\mathbf{x}_n\), \(\mathbf{x}'_n\), and \(m \in \mathcal{C}(n)\),
\begin{align*}
& \ \left|\sup_{\mathbb{P} \in \mathcal{P}_m(\mathbf{x}_n)} \mathbb{E_P}\left[\mathcal{Q}^{\mathbf{x,y}}_{m}(\mathbf{x}_{n},\Tilde{\mathbf{a}}_{m})\right] - \sup_{\mathbb{P} \in \mathcal{P}_m(\mathbf{x}'_n)} \mathbb{E_P}\left[\mathcal{Q}^{\mathbf{x,y}}_{m}(\mathbf{x}'_{n},\Tilde{\mathbf{a}}_{m})\right]\right| \\
\ \leq & \ \left|\sup_{\mathbb{P} \in \mathcal{P}_m(\mathbf{x}_n)} \mathbb{E_P}\left[\mathcal{Q}^{\mathbf{x,y}}_{m}(\mathbf{x}_{n},\Tilde{\mathbf{a}}_{m})\right] - \sup_{\mathbb{P} \in \mathcal{P}_m(\mathbf{x}'_n)} \mathbb{E_P}\left[\mathcal{Q}^{\mathbf{x,y}}_{m}(\mathbf{x}_{n},\Tilde{\mathbf{a}}_{m})\right]\right| + E_m\|\mathbf{x}_n - \mathbf{x}'_n\|_{\infty} \\
\leq & \ (H_m + E_m)\|\mathbf{x}_n - \mathbf{x}'_n\|_{\infty}
\end{align*}
for some \(H_m \geq 0\), where the second inequality follows from the Hoffman's Lemma. Then, for an optimal solution \((\mathbf{x}^*_n, \mathbf{y}^*_n)\) to the above formulation, there exists a feasible \((\overline{\mathbf{x}}_n, \overline{\mathbf{y}}_n) \in S(\mathbf{x}_n) \times S(\mathbf{y}_n)\) such that \(\|(\overline{\mathbf{x}}_n, \overline{\mathbf{y}}_n) - (\mathbf{x}^*_n, \mathbf{y}^*_n)\|_{\infty} \leq s\). It follows that
\begin{align*}
\mathcal{Q}^{\mathbf{x,y}}_{n}(\mathbf{x}_{p(n)},\Tilde{\mathbf{a}}_{n}) \ \leq & \ f_n(\overline{\mathbf{x}}_{n}, \overline{\mathbf{y}}_{n}, \tilde{\mathbf{a}}_n) + \sum_{m \in \mathcal{C}(n)} p_{nm} \sup_{\mathbb{P} \in \mathcal{P}_m(\overline{\mathbf{x}}_n)} \mathbb{E}_{\mathbb{P}}\left[ \mathcal{Q}^{\mathbf{x,y}}_{m}(\overline{\mathbf{x}}_{n},\Tilde{\mathbf{a}}_{m})\right] \\
\leq & \ f_n(\mathbf{x}^*_{n}, \mathbf{y}^*_{n}, \tilde{\mathbf{a}}_n) + Ls + \sum_{m \in \mathcal{C}(n)} p_{nm} \left(\sup_{\mathbb{P} \in \mathcal{P}_m(\mathbf{x}^*_n)} \mathbb{E}_{\mathbb{P}}\left[ \mathcal{Q}^{\mathbf{x,y}}_{m}(\mathbf{x}^*_{n},\Tilde{\mathbf{a}}_{m})\right] + (H_m + E_m)s \right) \\
\leq & \ f_n(\mathbf{x}^*_{n}, \mathbf{y}^*_{n}, \tilde{\mathbf{a}}_n) + \sum_{m \in \mathcal{C}(n)} p_{nm} \sup_{\mathbb{P} \in \mathcal{P}_m(\mathbf{x}^*_n)} \mathbb{E}_{\mathbb{P}}\left[ \mathcal{Q}_{m}(\mathbf{x}^*_{n},\Tilde{\mathbf{a}}_{m})\right] + (L + H_m + E_m)s \\
\leq & \ \mathcal{Q}_{n}(\mathbf{x}_{p(n)},\Tilde{\mathbf{a}}_{n}) + (L + H_m + E_m)s.
\end{align*}
Finally, since function \(\displaystyle\sup_{\mathbb{P} \in \mathcal{P}_m(\mathbf{x}_n)} \mathbb{E}_{\mathbb{P}}\left[ \mathcal{Q}^{\mathbf{x,y}}_{m}(\mathbf{x}_{n},\Tilde{\mathbf{a}}_{m})\right]\) is defined on a binary domain due to the binary expansion of \((\mathbf{x}_n,\mathbf{y}_n)\), its epigraph admits a polyhedral description. Then, the formulation of \(\mathcal{Q}^{\mathbf{x,y}}_{n}(\mathbf{x}_{p(n)},\Tilde{\mathbf{a}}_{n})\) can be written as a binary integer program. Hence, we can once again invoke the proximity result (Theorem 2.2 in~\cite{blair1977value}) to obtain
\begin{equation*}
    |\mathcal{Q}^{\mathbf{x,y}}_{n}(\mathbf{x}_{p(n)},\Tilde{\mathbf{a}}_{n}) - \mathcal{Q}^{\mathbf{x,y}}_{n}(\mathbf{x}'_{p(n)},\Tilde{\mathbf{a}}_{n})| \leq E_n\|\mathbf{x}_{p(n)} - \mathbf{x}'_{p(n)}\|_{\infty}, \; \forall \mathbf{x}_{p(n)}, \mathbf{x}'_{p(n)}
\end{equation*}
for some \(E_n \geq 0\). This completes the proof.
\end{proof}
Theorem~\ref{thm:expansion} implies that the binary expansion approximation with a precision \(s\) of the multi-stage DRO model converges linearly to the true optimal value as \(s\) decreases towards zero.
\subsection{Technical Proofs} \label{apdx:proofs}
\begin{proof}[\textbf{Proof of Proposition~\ref{prop:exp-rep}}]
We express the worst-case expectation with respect to \(\mathbb{P}\) as an optimization formulation:
\begin{subequations}
\begin{align}
\max_{\mathbf{p}_{\mathbf{a}_{m}}} \quad & \sum_{\mathbf{a}_{m} \in \mathcal{A}_{m}} p_{\mathbf{a}_{m}} \mathcal{Q}_{m}(\mathbf{x}_{n},\mathbf{a}_{m}) \\
        \text{s.t.} \quad &\sum\limits_{\mathbf{a}_{m} \in \mathcal{A}_{m}}p_{\mathbf{a}_{m}} \mathbf{a}_{m} \leq  \beta_{m}\mathbf{x}_{n} + \gamma_{m} & [\psi_{m}] \\
         &\sum\limits_{\mathbf{a}_{m}\in \mathcal{A}_{m}}p_{\mathbf{a}_{m}}  =  1 & [\phi_{m}] \label{eq:Ambiguity:LP-con-1} \\
         & p_{\mathbf{a}_{m}} \geq 0, \quad \forall \ \mathbf{a}_{m} \in \mathcal{A}_{m}. & \label{eq:Ambiguity:LP-con-2}
\end{align}
\label{eq:Ambiguity:LP}
\end{subequations}
In the above formulation, we encode \(\mathbb{P}\) using a finite-dimensional vector \(p_{\mathbf{a}_{m}}\) because \(\tilde{\mathbf{a}}_{m}\) is a discrete random vector with a finite support \(\mathcal{A}_m \subseteq \{0, 1\}^{|\mathcal{L}|}\). This renders formulation~\eqref{eq:Ambiguity:LP} a finite-dimensional linear program and the conclusion follows from the standard dual formulation, where $\bm{\psi}_{m}$ and $\bm{\phi}_{m}$ are dual variables associated with constraints~\eqref{eq:Ambiguity:LP-con-1}--\eqref{eq:Ambiguity:LP-con-2}, respectively.
\end{proof}
\begin{proof}[\textbf{Proof of Proposition~\ref{prop:separation}}]
By construction, we recast~\eqref{eq:MaxCut} as
\begin{align*}
\max_{\mathbf{a} \in \mathcal{A}_{m}}\{\mathcal{Q}_{m}(\mathbf{x}_{n},\mathbf{a}) - \mathbf{a}^{\top} \bm{\psi}_{m}\} \ = & \ \max_{\mathbf{a} \in \mathcal{A}_{m}}\left\{\max_{h \in \mathcal{H}_m} \left\{ (\pi^{h})^{\top}\mathbf{x}_{n} + (\tau^{h})^{\top}\mathbf{a} + \omega^{h} \right\} - \mathbf{a}^{\top} \bm{\psi}_{m}\right\} \\
= & \ \max_{h \in \mathcal{H}_m}\max_{\mathbf{a} \in \mathcal{A}_{m}}\left\{ (\pi^{h})^{\top}\mathbf{x}_{n} + (\tau^{h}-\bm{\psi}_{m})^{\top}\mathbf{a} + \omega^{h}\right\}.
\end{align*}
For each \(h \in \mathcal{H}_m\), the inner maximization problem with respect to \(\mathbf{a}\) is a linear integer program with an individual cardinality constraint. Hence, it can be solved by sorting the entries \((\tau^h - \bm\psi_m)\), as done in each for loop of Algorithm~\ref{alg:separation}. Therefore, Algorithm~\ref{alg:separation} correctly evaluates the optimal value of~\eqref{eq:MaxCut}. Finally, since Algorithm~\ref{alg:separation} conducts \(|\mathcal{H}_m|\) for loops and each loop sorts at most \(|\mathcal{L}|\) entries, the running time is \(O(|\mathcal{H}_m||\mathcal{L}|\log(|\mathcal{L}|))\).
\end{proof}

\begin{proof}[\textbf{Proof of Proposition~\ref{prop:worst-dist}}]
First, formulations~\eqref{model:Q_min} and~\eqref{model:Q_min-reduced} have the same optimal value by construction of the delayed constraint generation. The equality in~\eqref{wc-note-1} follows. Second, we take the dual of formulation~\eqref{model:Q_min-reduced} to obtain
\begin{align*}
    \sup_{\mathbb{P} \in \mathcal{P}_m(\mathbf{x}_n)} \mathbb{E}_{\mathbb{P}}\left[ \mathcal{Q}_{m}(\mathbf{x}_{n},\Tilde{\mathbf{a}}_{m})\right] \ = \ \max_{\bm{\lambda}\geq0} \quad & \sum_{\mathbf{a}^{h} \in \mathcal{A}^*_m} \lambda^h \mathcal{Q}_{m}(\mathbf{x}_{n},\mathbf{a}^{h}) \nonumber \\
    \text{s.t.} \quad & \sum_{\mathbf{a}^{h} \in \mathcal{A}^*_m}\lambda^h \mathbf{a}^{h} \leq \beta_{m}\mathbf{x}_{n} + \gamma_{m}, & \\
    &\sum_{\mathbf{a}^{h}\in \mathcal{A}^*_m} \lambda^h = 1, & 
\end{align*}
The last two constraints show that a dual optimal solution \(\lambda^h\) defines a probability distribution \(\mathbb{P}^*\) for \(\tilde{\mathbf{a}}_m\) supported on \(\mathcal{A}^*_m \subseteq \mathcal{A}_m\) and the first constraint shows that \(\mathbb{P}^* \in \mathcal{P}_m(\mathbf{x}_n)\). Finally, we complete the proof by noticing that
\begin{equation*}
    \sum_{\mathbf{a}^{h} \in \mathcal{A}^*_m} \lambda^h \mathcal{Q}_{m}(\mathbf{x}_{n},\mathbf{a}^{h}) = \mathbb{E}_{\mathbb{P}^*}[\mathcal{Q}_m(\mathbf{x}_n, \tilde{\mathbf{a}}_m)].
\end{equation*}
\end{proof}
\begin{proof}[\textbf{Proof of Proposition~\ref{prop:integer-cut}}]
It is clear that 
\begin{equation*}
\underline{\mathcal{Q}}_n(\mathbf{x}_{p(n)}, \tilde{\mathbf{a}}_n) \geq \underline{\mathcal{Q}}_n(\hat{\mathbf{x}}_{p(n)}, \hat{\mathbf{a}}_n) - \left(\underline{\mathcal{Q}}_n(\hat{\mathbf{x}}_{p(n)}, \hat{\mathbf{a}}_n) - L_n\right) \left(\|\mathbf{x}_{p(n)} - \hat{\mathbf{x}}_{p(n)}\|_1 + \|\tilde{\mathbf{a}}_n - \hat{\mathbf{a}}_n\|_1\right).
\end{equation*}
The conclusion follows from noticing that \(\|\mathbf{x}_{p(n)} - \hat{\mathbf{x}}_{p(n)}\|_1 = \mathbf{1}^{\top}\mathbf{x}_{p(n)} + \mathbf{1}^{\top}\hat{\mathbf{x}}_{p(n)} - 2\hat{\mathbf{x}}_{p(n)}^{\top}\mathbf{x}_{p(n)}\) and \(\|\tilde{\mathbf{a}}_n - \hat{\mathbf{a}}_n\|_1 = \mathbf{1}^{\top}\tilde{\mathbf{a}}_n + \mathbf{1}^{\top}\hat{\mathbf{a}}_n - 2\hat{\mathbf{a}}_n^{\top}\tilde{\mathbf{a}}_n\).
\end{proof}
\begin{proof}[\textbf{Proof of Proposition~\ref{prop:integer-cut-tight}}]
First, equality~\eqref{eq:tight-cut} holds for the integer optimality cut because
\begin{align*}
& \ \pi^{\top}\hat{\mathbf{x}}_{p(n)} + \tau^{\top}\hat{\mathbf{a}}_n + \omega \\
= \ & \ (\underline{\mathcal{Q}}_n(\hat{\mathbf{x}}_{p(n)}, \hat{\mathbf{a}}_n) - L_n)(2\hat{\mathbf{x}}_{p(n)} - \mathbf{1})^{\top} \hat{\mathbf{x}}_{p(n)} + (\underline{\mathcal{Q}}_n(\hat{\mathbf{x}}_{p(n)}, \hat{\mathbf{a}}_n) - L_n)(2\hat{\mathbf{a}}_n - \mathbf{1})^{\top} \hat{\mathbf{a}}_n + \underline{\mathcal{Q}}_n(\hat{\mathbf{x}}_{p(n)}, \hat{\mathbf{a}}_n) \\
& \ - (\underline{\mathcal{Q}}_n(\hat{\mathbf{x}}_{p(n)}, \hat{\mathbf{a}}_n) - L_n) (\mathbf{1}^{\top}\hat{\mathbf{x}}_{p(n)} + \mathbf{1}^{\top}\hat{\mathbf{a}}_n) \\
= \ & \ 2 (\underline{\mathcal{Q}}_n(\hat{\mathbf{x}}_{p(n)}, \hat{\mathbf{a}}_n) - L_n) (\hat{\mathbf{x}}_{p(n)} - \mathbf{1})^{\top}\hat{\mathbf{x}}_{p(n)} + 2 (\underline{\mathcal{Q}}_n(\hat{\mathbf{x}}_{p(n)}, \hat{\mathbf{a}}_n) - L_n) (\hat{\mathbf{a}}_n - \mathbf{1})^{\top}\hat{\mathbf{a}}_n + \underline{\mathcal{Q}}_n(\hat{\mathbf{x}}_{p(n)}, \hat{\mathbf{a}}_n) \\
= \ & \ \underline{\mathcal{Q}}_n(\hat{\mathbf{x}}_{p(n)}, \hat{\mathbf{a}}_n),
\end{align*}
where the last equality is because \((x - \mathbf{1})^{\top}x = 0\) for any binary vector \(x\). Second, the validity of the Lagrangian cut implies that \(\underline{\mathcal{Q}}_n(\hat{\mathbf{x}}_{p(n)},\hat{\mathbf{a}}_n) \geq \pi^{\top}\hat{\mathbf{x}}_{p(n)} + \tau^{\top}\hat{\mathbf{a}}_n + \mathcal{L}_n(\pi, \tau)\) for all \((\pi, \tau)\). But this inequality holds with equality for the choice of \((\pi, \tau)\) in the integer optimality cut, implying that \((\pi, \tau) \in \text{argmax}_{\pi, \tau}\{\pi^{\top}\hat{\mathbf{x}}_{p(n)} + \tau^{\top}\hat{\mathbf{a}}_n + \mathcal{L}_n(\pi, \tau)\}\). Finally, for notation brevity we denote \[\overline{f}_{n}(\mathbf{x}_{n},\mathbf{y}_{n},\mathbf{w}_n) := f_{n}(\mathbf{x}_{n},\mathbf{y}_{n},\mathbf{w}_n) + \sum_{m \in \mathcal{C}(n)} p_{nm} \sup_{\mathbb{P}\in \mathcal{P}_m(\mathbf{x}_n)} \mathbb{E_P}\left[ \underline{\mathcal{Q}}_m(\mathbf{x}_{n},\Tilde{\mathbf{a}}_{m})\right]\]. Then, 
\begin{align*}
& \ \mathcal{L}_n(\pi, \tau) + \pi^{\top}\hat{\mathbf{x}}_{p(n)} + \tau^{\top}\hat{\mathbf{a}}_n \\
= \ & \min_{\substack{\mathbf{x}_n, \mathbf{y}_n, \\ \mathbf{r}_n,\mathbf{w}_n: \text{\eqref{model:Q_n-con}, \eqref{binary-con}}}} \ \overline{f}_{n}(\mathbf{x}_n,\mathbf{y}_n,\mathbf{w}_n) - \pi^{\top}(\mathbf{r}_{n} - \hat{\mathbf{x}}_{p(n)}) - \tau^{\top}(\mathbf{w}_{n} - \hat{\mathbf{a}}_n)\\
= \ & \min_{\substack{\mathbf{x}_n, \mathbf{y}_n, \\ \mathbf{r}_n,\mathbf{w}_n: \text{\eqref{model:Q_n-con}, \eqref{binary-con}}}} \ \overline{f}_{n}(\mathbf{x}_{n},\mathbf{y}_{n},\mathbf{w}_n) - (\underline{\mathcal{Q}}_n(\hat{\mathbf{x}}_{p(n)}, \hat{\mathbf{a}}_n) - L_n) (2\hat{\mathbf{x}}_{p(n)} - \mathbf{1})^{\top}(\mathbf{r}_{n} - \hat{\mathbf{x}}_{p(n)}) \\
& \hspace{6em} - (\underline{\mathcal{Q}}_n(\hat{\mathbf{x}}_{p(n)}, \hat{\mathbf{a}}_n) - L_n) (2\hat{\mathbf{a}}_n - \mathbf{1})^{\top}(\mathbf{w}_{n} - \hat{\mathbf{a}}_n)\\
= \ & \min_{\substack{\mathbf{x}_n, \mathbf{y}_n, \\ \mathbf{r}_n,\mathbf{w}_n: \text{\eqref{model:Q_n-con}, \eqref{binary-con}}}} \ \overline{f}_{n}(\mathbf{x}_{n},\mathbf{y}_{n},\mathbf{w}_n) + (\underline{\mathcal{Q}}_n(\hat{\mathbf{x}}_{p(n)}, \hat{\mathbf{a}}_n) - L_n) \|\hat{\mathbf{x}}_{p(n)} - \mathbf{r}_{n}\|_1 \\
& \hspace{6em} + (\underline{\mathcal{Q}}_n(\hat{\mathbf{x}}_{p(n)}, \hat{\mathbf{a}}_n) - L_n) \|\hat{\mathbf{a}}_n - \mathbf{w}_{n}\|_1 \\
= \ & \underline{\mathcal{Q}}_n(\hat{\mathbf{x}}_{p(n)}, \hat{\mathbf{a}}_n),
\end{align*}
where the second equality follows from the definition of \(\pi\) and \(\tau\), the third equality is because \((2x - \mathbf{1})^{\top}(y - x) = 2x^{\top}y - \mathbf{1}^{\top}x - \mathbf{1}^{\top}y = -\|x - y\|_1\) whenever \(x, y\) are binary vectors, and the last equality is because the objective function is at least \(L_n + (\underline{\mathcal{Q}}_n(\hat{\mathbf{x}}_{p(n)}, \hat{\mathbf{a}}_n) - L_n) = \underline{\mathcal{Q}}_n(\hat{\mathbf{x}}_{p(n)}, \hat{\mathbf{a}}_n)\) if either \(\mathbf{w}_{n} \neq \hat{\mathbf{a}}_n\) or \(\mathbf{r}_{n} \neq \hat{\mathbf{x}}_{p(n)}\), implying that \(\mathbf{w}_{n} = \hat{\mathbf{a}}_n\) and \(\mathbf{r}_{n} = \hat{\mathbf{x}}_{p(n)}\) at optimum. It follows that \(\mathcal{L}_n(\pi, \tau) = \underline{\mathcal{Q}}_n(\hat{\mathbf{x}}_{p(n)}, \hat{\mathbf{a}}_n) - \pi^{\top}\hat{\mathbf{x}}_{p(n)} - \tau^{\top}\hat{\mathbf{a}}_n = \omega\), which finishes the proof. 
\end{proof}
\subsection{Two-Stage Line Switching Model}
\label{subsec:MultiScaleDecisions}
Within the SND algorithm, the complete set of variables is decided at every stage $t \in [T]$. We can consider a more general decision setup, where a set of variables is decided at the root node (first stage), and implemented in the later stages, with other variables being decided at every stage. For exposition purposes, we will consider a problem with two-stage line switching decisions and dynamic power generation and power flow. Thus, we will only consider a set of line switching decisions for each stage $t \in [T]$ that will be used at nodes in the stage node set $\mathcal{S}_{t}$.

We can modify the current formulation by extending the state variable vector $\mathbf{x}$ for the root node at stage $t=1$ to considering all switching decisions in the planning horizon. 

Given that we  obtaining dual information of every stage switching, we propagate the dual information beyond subsequent stages to plan the switching for all stages at the first stage. We therefore, can consider the use of additional fishing variables and constraints for this propagation. In this way, in the forward pass, at the root node, we obtain the switching decisions for next stages, and when solving the sampled paths, we fix all switching decisions beyond one later stage, as done in the standard SND algorithm. At the backward step, at any node $n \in \mathcal{T}\setminus\{1\}$, we solve the problem:
\begin{subequations}
    \label{model:Q_hybrid}
    \begin{align}\min_{\mathbf{x}_{n},\mathbf{s}_{t},\mathbf{y}_{n},\psi_{m},\phi_{m},\mathbf{r},\mathbf{w}} \quad & f_{n}(\mathbf{x}_{n},\mathbf{y}_{n}) + \sum_{m \in \mathcal{C}(n)} q_{nm} \left( \psi_{m}^{\top} \beta_{m}\mathbf{x}_{n} + \psi_{m}^{\top}\gamma_{m} + \phi_{m} \right) \nonumber\\
         \text{s.t.}\quad & A_n \mathbf{x}_{n} + W_{n}\mathbf{y}_{n} + C_{n} \mathbf{x}_{p(n)} + D_{n}\Tilde{\mathbf{a}}_{n} \geq h_{n} \notag \\
         & \mathbf{r}_{n} = \mathbf{x}_{n}\\
         & \mathbf{s}_{i} = \mathbf{z}_{i}, \forall i = t,\dots,T\\
         & \mathbf{w}_{n} = \Tilde{\mathbf{a}}_{n}\\
         & \phi_{m} \geq \max_{\mathbf{a}} \left \{(\pi^{h}_{m})^{\top}\mathbf{x}_{n} + \sum_{i=t}^{T} (\eta_{m}^{h})^{\top}\mathbf{s}_{i} + (\tau_{m}^{h}-\psi_{m})^{\top}\mathbf{a} + \omega_{m}^{h} \right\}, \forall h \in \mathcal{H}.\label{cst:Q_hybrid:cuts}
    \end{align}
\end{subequations}
\subsection{Initial Upper Bound Computation}
\label{apdx:UpperBound}

Algorithm \ref{alg:SDN} requires a big-$\mathcal{M}$ coefficient that serves as an upper bound to have a valid deterministic upper bound via vertex enumeration. We can consider the worst-case realization of the line survival vector as a tighter initial upper bound. Having all failure-prone lines to fail is the worst-case realization for the uncertainty. During the first iteration of Algorithm \ref{alg:SDN}, we run a forward pass by setting all failure-prone lines in $\mathbf{a}$, unavailable. We note that we construct a scenario tree that encodes the nodal information of the ambiguity set, but does not define the values of the matrices $A,W,C,D$ or vector $h$, which are only stage-dependent. Thus, if we consider the summation of the objective over all stages in this deterministic forward pass, we can set it as the big-$\mathcal{M}$ for all nodes in the corresponding stage. Furthermore, we can employ these trial points to generate an initial set of hyperplanes to initialize the outer approximation of the cost-to-go function for all nodes in the corresponding stage. Our numerical experiments use this procedure to compute this worst-case upper bound during the first iteration of the algorithm.

\subsection{Construction of easy-to-implement policies}
\label{apdx:Policies}

\begin{description}
    \item[Topology policy] For each time unit \(t \in [T]\), the \emph{Topology} policy chooses a transmission grid topology from a set of topologies. Algorithm~\ref{alg:Topology} describes how to obtain the set of topologies through the SND algorithm.

\begin{algorithm}
\small
    \caption{Construction of topology sets for the \emph{Topology} policy}
    \label{alg:Topology}
    \begin{algorithmic}[1]
    \State {\bf Input}: optimal solutions $\{\mathbf{z}^{*}_{n}, n \in \mathcal{T}\}$ of the SND algorithm, $\text{Top}_{t} \gets \emptyset$ for all $t \in [T]$;
    \For{\(t = 1, \ldots, T\)}
        \State Sort the nodes \(\mathcal{S}_t\) in a decreasing order of their probabilities of occurrence and denote the permutation as \(<1>, \ldots, <|\mathcal{S}_t|>\);
        \For{\(n = 1, \ldots, |\mathcal{S}_t|\)}
            \If{$\text{Top}_{t} = \emptyset$ or $\min_{\mathbf{z} \in \text{Top}_{t}}\{||\mathbf{z} - \mathbf{z}_{<n>}^{*}||_{1}/\|\mathbf{z}\|_1\} \geq 0.05$}
            \State $\text{Top}_{t} \gets \text{Top}_{t} \cup \{z_{<n>}^{*}\}$;
        \EndIf
        \EndFor
    \EndFor
    \State Return $\text{Top}_{t}$ for all $t \in [T]$;
    \end{algorithmic}
\end{algorithm}
Intuitively, Algorithm~\ref{alg:Topology} greedily collects a subset of optimal grid topologies from the SND algorithm. We require that these topologies are sufficiently different from each other to obtain a diverse candidate set for the \emph{Topology} policy. Figure~\ref{fig:Topologies} depicts a set of three topologies used in this policy for the grid in Figure~\ref{fig:SolutionViz}.
\begin{figure}[ht]
\centering
\begin{subfigure}[b]{0.25\textwidth}
    \centering
    \includegraphics[width=\textwidth]{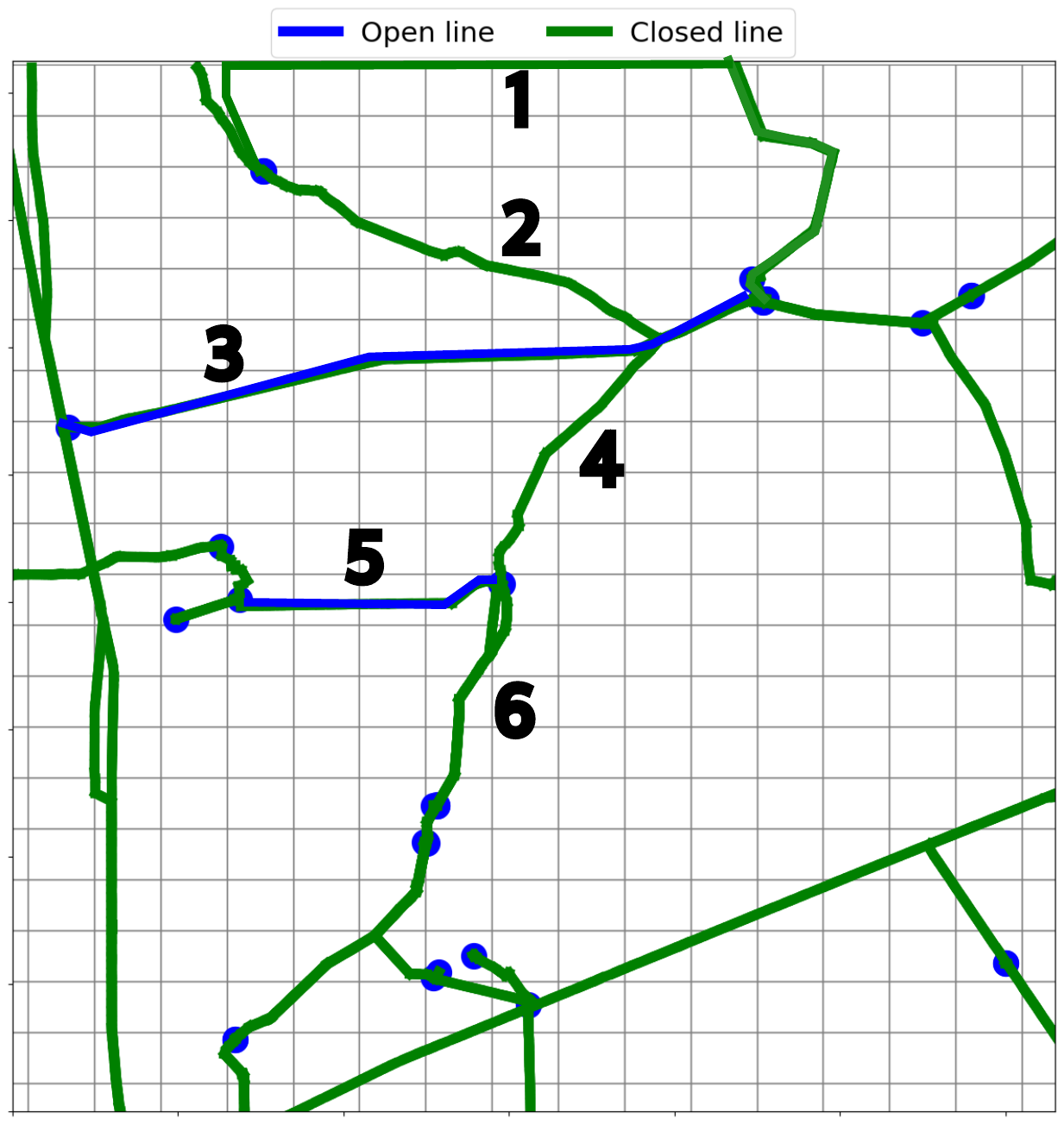}
    \caption{Topology 1}
    \label{fig:Top1}
\end{subfigure}
\begin{subfigure}[b]{0.25\textwidth}
    \centering
    \includegraphics[width=\textwidth]{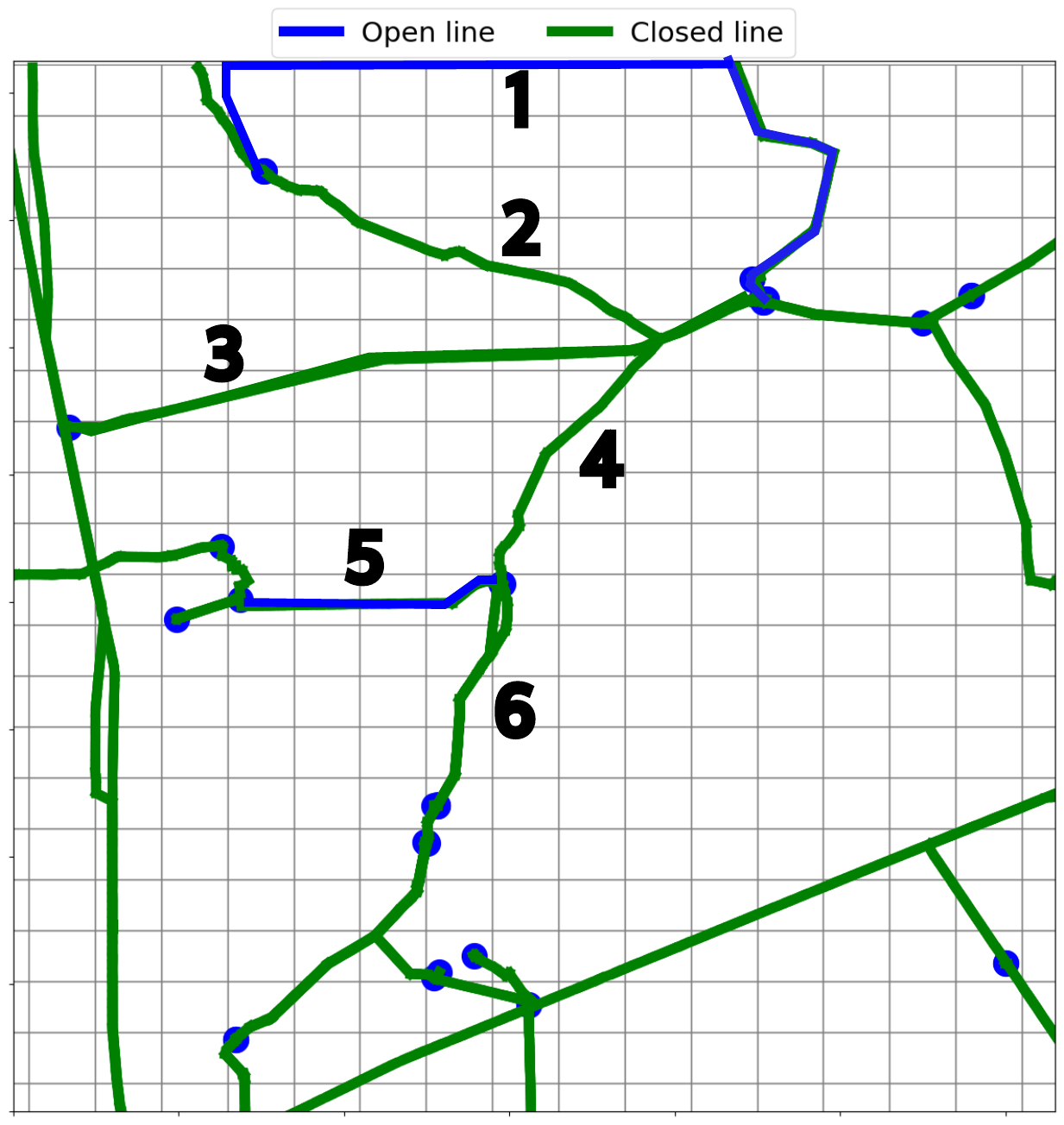}
    \caption{Topology 2}
    \label{fig:Top2}
\end{subfigure}
\begin{subfigure}[b]{0.25\textwidth}
    \centering
    \includegraphics[width=\textwidth]{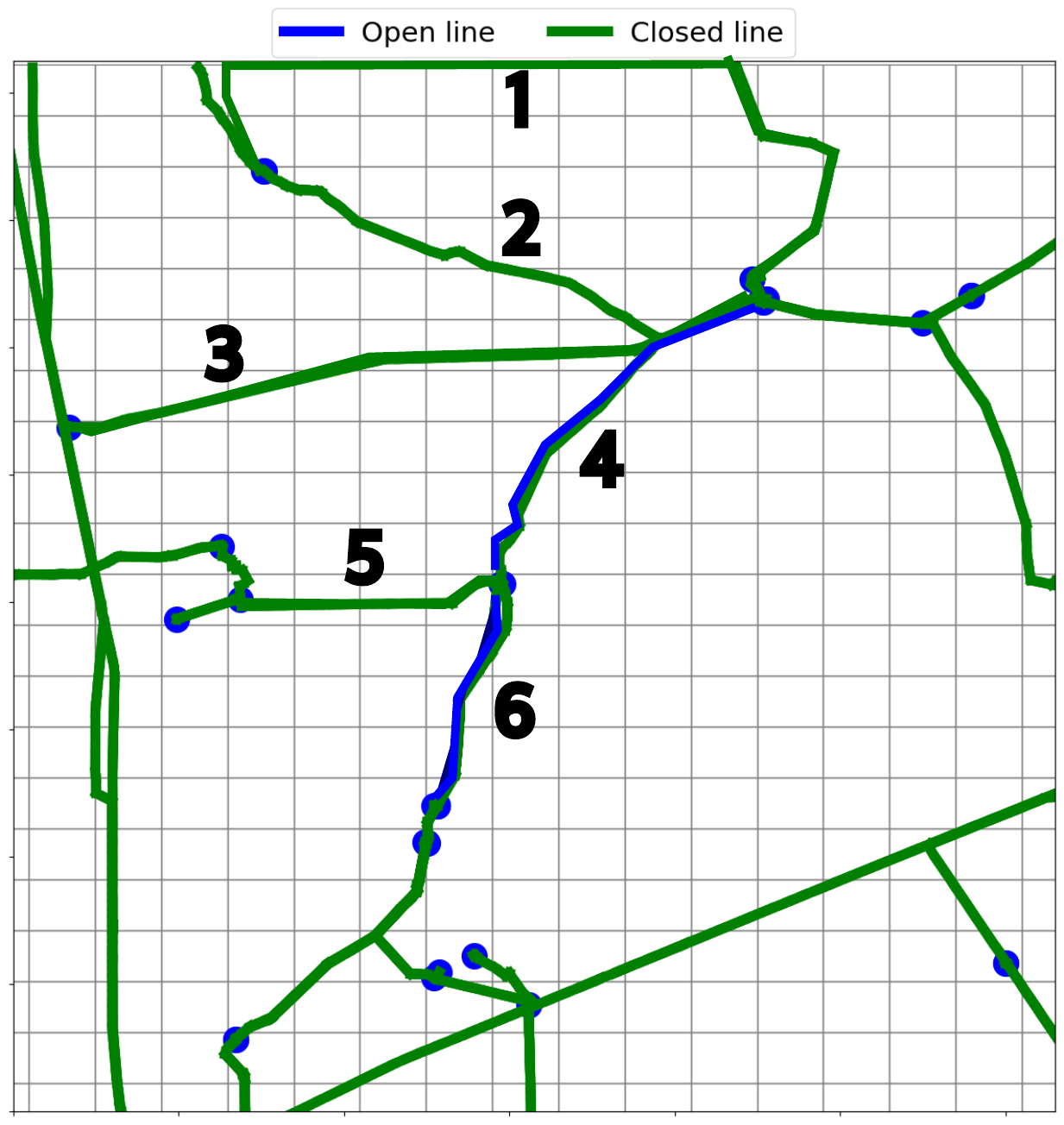}
    \caption{Topology 3}
    \label{fig:Top3}
\end{subfigure}
\caption{Example topologies for subregion in Figure \ref{fig:SolutionViz}}
\label{fig:Topologies}
\end{figure}
\item[Mapping policy] For each time unit \(t \in [T]\) and for each switchable line \(l \in \mathcal{L}\), the \emph{Mapping} policy seeks to establish a deterministic look-up table mapping to the switching decision \(z_{lt}\) from the current state \(z_{l,t-1}\) of line \(l\), as well as the availability \(\tilde{\mathbf{a}}_t^{\mathcal{L}(l)}\) of the redundant lines in \(\mathcal{L}(l)\). We establish the \emph{Mapping} policy by sampling the scenario tree in multiple replications and, in each replication, retrieving the \emph{Dynamic} policy for \(z_{lt}\) under different scenarios but the same \((z_{l,t-1}, \tilde{\mathbf{a}}_t^{\mathcal{L}(l)})\) input. Then, we round the average of these \(z_{lt}\) to either zero or one. We detail this approach in Algorithm~\ref{alg:Mapping}.
\end{description}
\begin{algorithm}
\small
    \caption{\emph{Mapping} policy for transmission line \(l\)}
    \label{alg:Mapping}
    \begin{algorithmic}[1]
    \State Input: set of redundant lines \(\mathcal{L}(l)\), number of replications \(I\);
    \For{$i = 1,\dots,I$}
        \State Sample a scenario $\omega^i$ from \(\mathcal{T}\);
        \For{$t = 1,\dots,T$}
            \State Solve formulation~\eqref{model:Q_n} pertaining to node \(n \in \mathcal{S}_t \cap \omega^i\) and state \((\mathbf{x}^i_{t-1},\tilde{\mathbf{a}}^i_t)\) with $\mathcal{Q}_{m}(\cdot, \cdot)$ replaced by $\underline{\mathcal{Q}}_{m}(\cdot,\cdot)$ and store solution $(\mathbf{x}^i_{t},\mathbf{y}^i_t)$;
            \State Sample $\tilde{\mathbf{a}}^i_{t+1}$ from the worst-case distribution $\mathbb{P}^{*}$; 
            \For{ all $S \in \{0,1\} \times 2^{\mathcal{L}(l)}$}
                \State Construct a state \((\mathbf{x}'_{t-1},\Tilde{\mathbf{a}}'_{t})\) by replacing the entries of \((z_{l,t-1}, \tilde{\mathbf{a}}^{\mathcal{L}(l)})\) in $(\mathbf{x}^i_{t-1},\Tilde{\mathbf{a}}^i_{t})$ by \(S\);
                \State Solve formulation~\eqref{model:Q_n} pertaining to \((\mathbf{x}'_{t-1},\tilde{\mathbf{a}}'_t)\) with $\mathcal{Q}_{m}(\cdot, \cdot)$ replaced by $\underline{\mathcal{Q}}_{m}(\cdot,\cdot)$ and store the line switching decision \(z^i_{lt}(S)\);
            \EndFor
        \EndFor
    \EndFor
    \State Return mapping function \(\text{Map}_{l,t}: \{0,1\} \times 2^{\mathcal{L}(l)} \rightarrow \{0,1\}\) with $\text{Map}_{l,t}(S) := \text{Round}(\sum_{i=1}^Iz^i_{lt}(S)/I)$ for all \(t \in [T]\);
    \end{algorithmic}
\end{algorithm}
\subsection{Instance details}
\label{apdx:Instance}
\begin{figure}[ht]
\centering
\begin{subfigure}[b]{0.25\textwidth}
    \centering
    \includegraphics[width=\textwidth]{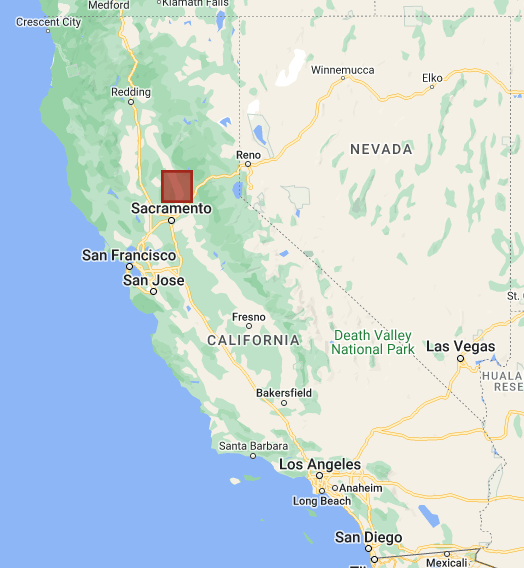}
    \caption{Selected area \label{fig:AreaMap}}
\end{subfigure}
\begin{subfigure}[b]{0.33\textwidth}
    \centering
    \includegraphics[width=\textwidth]{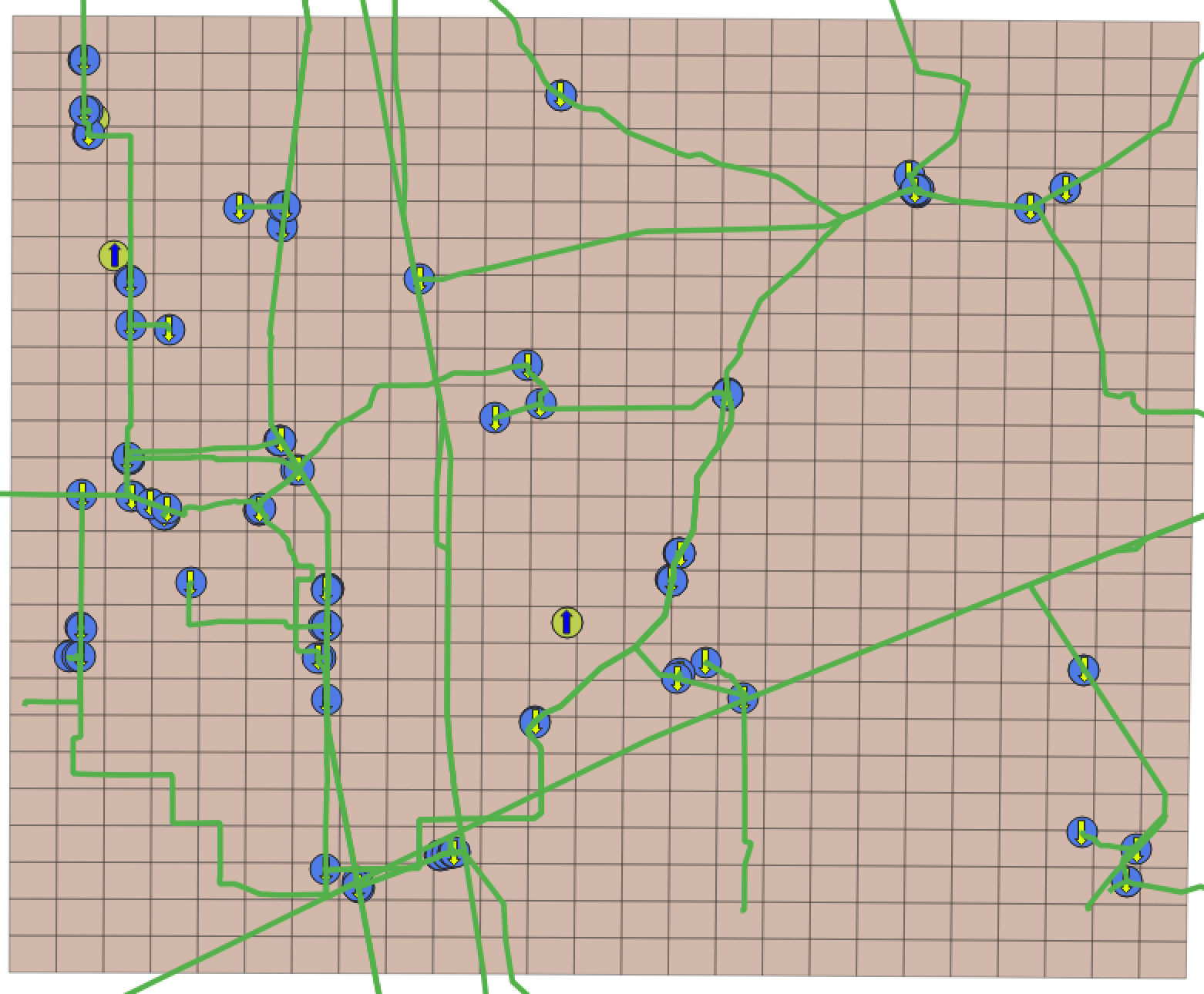}
    \caption{Transmission grid in selected area\label{fig:GridwBuses}}
\end{subfigure}
\caption{Area and topology of test instance \label{fig:Instance1}. Buses are marked in blue circles, generator in green circles, and transmission lines are marked in green.}
\end{figure}
\begin{figure}[ht]
\centering
\begin{subfigure}[b]{0.3\textwidth}
    \centering
    \includegraphics[width=\textwidth]{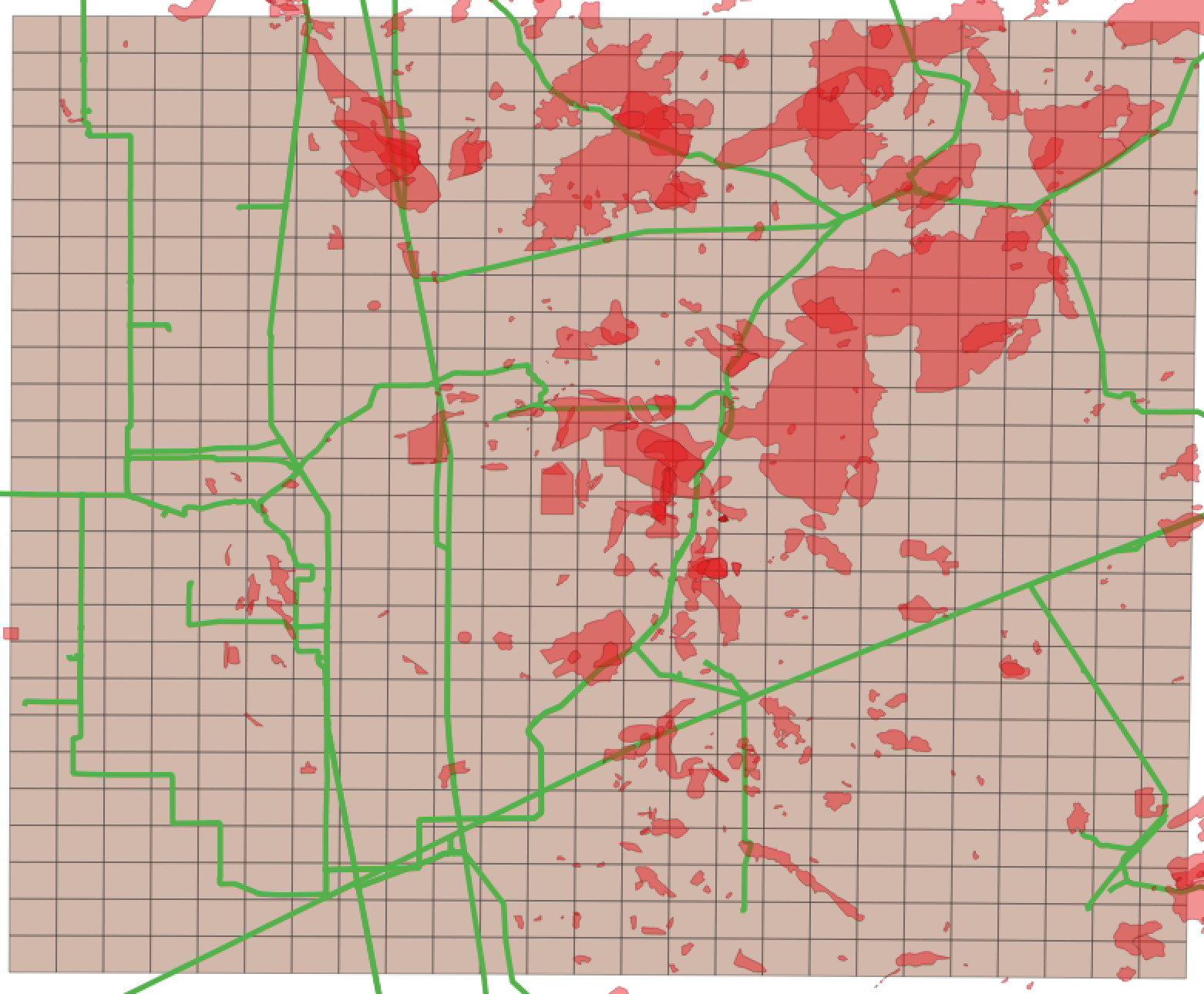}
    \caption{Historical wildfire perimeters \label{fig:WildfireMap}}
\end{subfigure}
\begin{subfigure}[b]{0.4\textwidth}
    \centering
    \includegraphics[width=\textwidth]{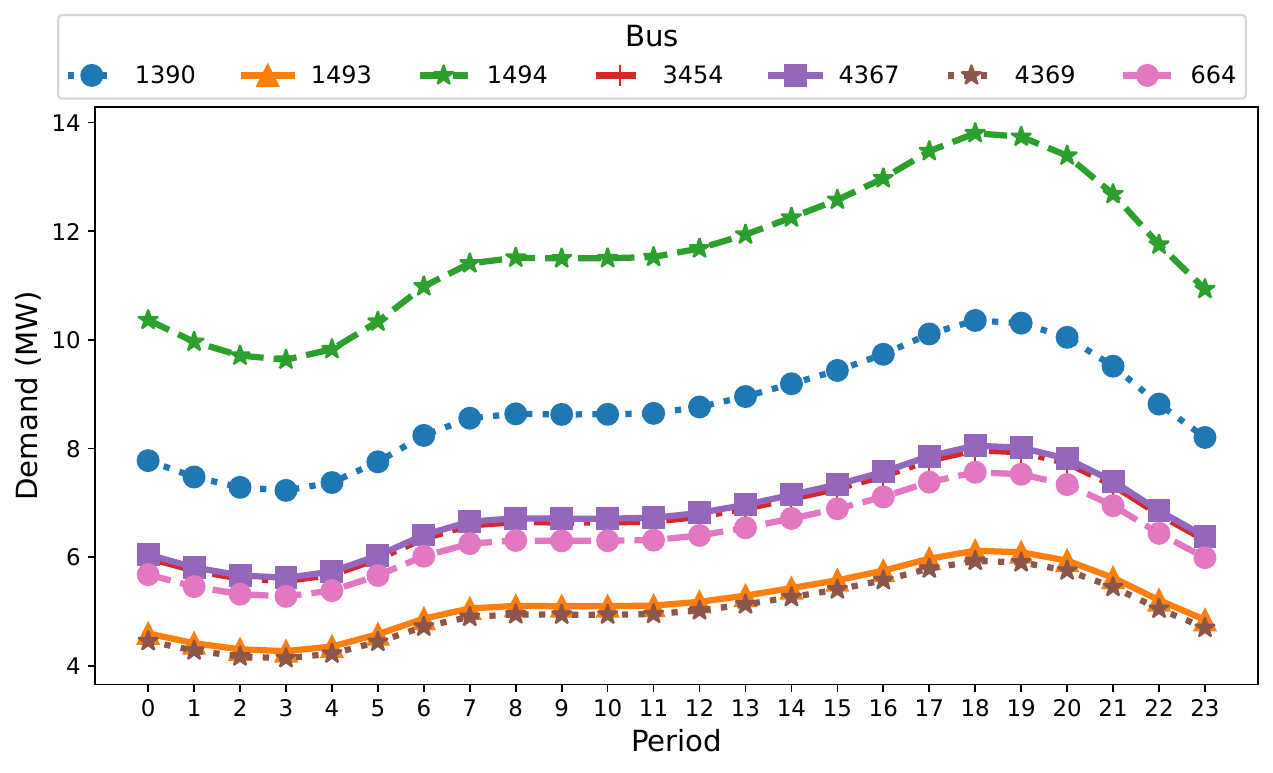}
    \caption{Selected area representative load curves \label{fig:LoadCurves}}
\end{subfigure}
\caption{Wildfire and load data for selected area \label{fig:Instance2}}
\end{figure}

\end{document}